\documentclass[10pt,reqno]{amsart}
\usepackage[top=1.2in, left=1.1in, right=1.1in, bottom=.9in]{geometry}

\title{A sharp quantitative stability result near infinitely concentrated minimisers}
\author{Melanie Rupflin and Sebastian Woodward}
\date{\today}
\usepackage[british,american]{babel}

\usepackage{mathrsfs}
\usepackage{amsfonts, amsmath, wasysym}
\usepackage{comment}

\usepackage{enumitem}
\usepackage{amssymb,amsthm}
\usepackage{nccmath}

\usepackage{todonotes}
\newcounter{todocounter}

\usepackage{
latexsym,
nicefrac,
}
\usepackage{url}

\definecolor{darkgreen}{rgb}{0,0.5,0}
\definecolor{darkred}{rgb}{0.7,0,0}
\usepackage[colorlinks, 
citecolor=darkgreen, linkcolor=darkred,urlcolor=black
]{hyperref}
\usepackage{esint}
\usepackage{bibgerm}
\usepackage[normalem]{ulem}

\theoremstyle{plain}
\newtheorem{lemma}{Lemma}[section]
\newtheorem{thm}[lemma]{Theorem}
\newtheorem{prop}[lemma]{Proposition}
\newtheorem{cor}[lemma]{Corollary}

\theoremstyle{definition}

\newtheorem{rmk}[lemma]{Remark}
\renewcommand\labelenumi{(\roman{enumi})}
\renewcommand\theenumi\labelenumi

\numberwithin{equation}{section}

\newcommand{\al}{\alpha}

\newcommand{\Ga}{\Gamma}
\newcommand{\de}{\delta}
\newcommand{\om}{\omega}
\newcommand{\Om}{\Omega}

\newcommand{\La}{\Lambda}
\newcommand{\rh}{\rho}
\newcommand{\si}{\sigma}

\newcommand{\Si}{\Sigma}
\renewcommand{\th}{\theta}

\newcommand{\ph}{\phi}

\newcommand{\ep}{\varepsilon}

\newcommand{\peps}{\partial_{\eps}}

\newcommand{\pmu}{\partial_\mu}
\newcommand{\R}{\ensuremath{{\mathbb R}}}
\newcommand{\N}{\ensuremath{{\mathbb N}}}

\newcommand{\C}{\ensuremath{{\mathbb C}}}

\newcommand{\err}{\text{err}}

\newcommand{\upto}{\uparrow}

\DeclareMathOperator{\inj}{inj}

\DeclareMathOperator{\id}{id}

\newcommand{\norm}[1]{\Vert#1\Vert} 

\def\osc{\mathop{{\mathrm{osc}}}\limits}

\newcommand{\beq}{\begin{equation}}
\newcommand{\eeq}{\end{equation}}
\newcommand{\beqs}{\begin{equation*}}
\newcommand{\eeqs}{\end{equation*}}
\newcommand{\beqa}{\begin{equation}\begin{aligned}}
\newcommand{\eeqa}{\end{aligned}\end{equation}}
\newcommand{\beqas}{\begin{equation*}\begin{aligned}}
\newcommand{\eeqas}{\end{aligned}\end{equation*}}
\newcommand{\brmk}{\begin{rmk}}
\newcommand{\ermk}{\end{rmk}}
\newcommand{\partref}[1]{\hbox{(\csname @roman\endcsname{\ref{#1}})}}
\newcommand{\half}{\frac{1}{2}}
\newcommand{\thalf}{\tfrac{1}{2}}

\newcommand{\na}{\nabla}

\newcommand{\pt}{\partial_t}
\newcommand{\abs}[1]{\vert#1\vert} 

\newcommand{\dist}{\text{dist}}

\newcommand{\supp}{\text{supp}} 
\newcommand{\eps}{\varepsilon}

\newcommand{\leqs}{\lesssim}
\newcommand{\geqs}{\gtrsim}
\newcommand{\Loj}{{\L}ojasiewicz }

\newcommand{\DD}{\mathbb{D}}

\newcommand{\ddt}{\tfrac{d}{dt}}

\newcommand{\la}{\lambda}
\newcommand{\Area}{\text{Area}}

\newcommand{\ps}{\partial_s}
\newcommand{\ddeps}{\tfrac{d}{d\eps}}
\newcommand{\ddelta}{\tfrac{d}{d\delta}}
\newcommand{\ppeps}{\partial_\eps}

\newcommand{\GG}{{\mathcal{G}}}
\newcommand{\ZZ}{{\mathcal{Z}}}

\newcommand{\RR}{{\mathcal{R}}}
\newcommand{\MM}{{\mathcal{M}}}
\newcommand{\EE}{{\mathcal{E}}}
\newcommand{\PP}{{\mathcal{P}}}

\newcommand{\barg}{{\bar g}}
\newcommand{\gbarmb}{{\bar g_{\mu,b}}}
\newcommand{\gbarla}{{\bar g_{\la,a}}}
\newcommand{\gmb}{{ g_{\mu,b}}}

\newcommand{\zbarla}{\bar z_{\la,a}}

\newcommand{\honebarmb}{{H^1(\Si,\gbarmb)}}

\newcommand{\hatja}{\widehat{j_a}}
\newcommand{\zmbla}{z_{\mu,b}^{\la,a}}

\newcommand{\gcyl}{{g_{\text{cyl}}}}

\usepackage{dsfont}
\makeatletter
\newcommand{\myitem}[1]{
\item[#1]\protected@edef\@currentlabel{#1}
}
\makeatother
\begin{document}

\begin{abstract}
We consider the question of quantitative stability of minimisers 
for a well-known variational problem for which the infimum of the energy is not achieved in the classical sense,
namely for the Dirichlet energy of degree $1$ maps from closed surfaces $(\Si,g_{\Si})$  of positive genus into the unit sphere $S^2\subset \R^3$.
For this variational problem it is natural to view 
configurations which consist of a constant map from the given domain and an infinitely concentrated rotation as generalised minimisers and to hence ask whether the distance of almost minimisers $v:\Si\to S^2$ to this set of infinitely concentrated minimisers can be controlled in terms of the energy defect $\de_v=E(v)-\inf E=E(v)-4\pi$. 

In this paper we develop a new dynamic approach that allows us to change the topology of the domain in a well controlled manner and to deform almost minimising maps from surfaces of general genus into harmonic maps from the sphere in a way that yields sharp quantitative estimates on all key features that characterise the distance to the set of infinitely concentrated minimisers, i.e. the scale of concentration, the $H^1$-distance to the nearest bubble on the concentration region and the $H^1$-distance to the nearest constant away from the concentration point. 
\end{abstract}
\maketitle

\section{Introduction and main results}\label{sec:intro}
Let $(\Si,g_\Si)$ be a closed orientable surface. We consider the Dirichlet energy 
 \beq
 \label{def:E-Dir-map}
 E(v)=E(v,g_\Si):= \tfrac{1}{2} \int_{\Si} \abs{\na_{g_\Si} v}_{g_\Si}^2 \text{d}v_{g_\Si}
 \eeq
 of maps $v:(\Si,g_\Si)\to S^2$. 
We recall that $E$ is invariant under conformal changes of the domain metric and isometries on the target. In the following, we can hence restrict our attention to metrics $g_\Si$ with curvature $K_{g_\Si} \equiv 1,0,-1$ for surfaces of genus $\gamma=0,1$ and $\geq 2$ respectively, additionally normalised by $\text{Area}(\Si,g_\Si) = 1$ if $\gamma=1$, and maps with non-negative degree.

  We also recall that the Dirichlet energy of maps from surfaces into arbitrary targets $N$ 
 is always bounded from below by the 
  area (counted with multiplicity), with equality if and only if the map is (weakly) conformal, i.e. if and only if $v^*g_{N}=\rho^2 g_{\Si}$ for a function $\rho\geq 0$.

In particular, for maps $v:(\Si,g_\Si)\to S^2$ of a given degree $k\in \N_0$ we have
$E(v)\geq 4\pi k $ 
with equality if and only if $v$ is conformal. 
Furthermore, we always have
\beq
\label{eq:inf=4pik}
\inf\{E(v,g_\Si):\quad  v:\Si \to S^2 \text{ with } \deg(v)=k
\}=4\pi k.  
\eeq
Indeed, writing for short $I=I(g_\Si):=\inj(\Si,g_\Si)$ and denoting by 
\beq \label{def:pi-mu}
\pi:\R^2\to S^2, \quad 
 \pi(x)= \big( \tfrac{2x}{1+\abs{x}^2}, \tfrac{1-\abs{x}^2}{1+\abs{x}^2} \big) \text{ and } \pi_\mu(\cdot)=\pi(\mu \cdot), \quad \mu>0\eeq
 the (rescaled) inverse stereographic projection, we can e.g. obtain a sequence of degree $k$ maps $v_n:\Si\to S^2$ with $E(v_n)\to 4\pi k$ by 
 working in local conformal coordinates $(x_1,x_2)$ on a ball $B_I(a)$, $a\in \Si$,
taking highly scaled copies $\pi_{\mu_n}(z^k)$, $\mu_n\to \infty$, of the degree $k$ harmonic map $\C\ni z\mapsto \pi(z^k)\in S^2$ 
and cutting them off to a constant on $B_{I}(a)\setminus B_{\half I}(a)$.  Here, and in the following, we identify $\R^2$ and $\C$ whenever convenient.
\vspace{0.1cm}

Conversely, the question of whether the minimum is achieved, i.e. whether the set of (classical) minimisers 
\beq
\label{def:Min-Set}
\MM_k(\Si,g_\Si) := \{ v : (\Si,g_\Si) \to (S^2,g_{S^2}) \text{ conformal with } \deg(v)=k  \}
\eeq 
is non-empty, depends on both the domain  $(\Si,g_\Si)$ and the prescribed degree $k$.

If $\Si$ is a sphere, then this set can be identified with 
the set of meromorphic functions from $\hat \C=\C\cup\{\infty\}$ to itself by working in stereographic coordinates on both the domain and the target, so $\dim(\MM_k(S^2))=4k+2>0$ for any $k\in \N_0$.

Conversely, 
the Riemann-Hurwitz formula excludes the existence of degree $1$ meromorphic functions from any surface $\Si$ of positive genus to $S^2$ so for all such domains the set of classical minimisers $\MM_1(\Si,g_\Si)$ is empty and 
\beq 
\label{def:energy-defect}
\de_v:=E(v)-4\pi>0 \text{ for every degree }1 \text{ map } v:(\Si,g_\Si)\to S^2.
\eeq
At the same time, we can view configurations which are given by a constant map $\om_0:\Si\to S^2$ away from a single point $a\in \Si$ and there described by an infinitely concentrated copy of a rotation $\RR\in SO(3)$ as generalised, infinitely concentrated minimisers of the above problem. 

It is hence natural to ask whether degree $1$ maps $v:\Si\to S^2 $ which almost minimise the energy must always be close to such a singular minimising configuration.

It is easy 
to answer this question positively at a qualitative level, i.e. to show that 
sequences of  maps $v_n$ whose energy defect 
$\de_{v_n}=E(v_n)-4\pi$ tends to zero subconverge (in a suitably defined $H^1$ sense) to such a singular minimiser. 

In the present paper we study the more challenging question of \textit{quantitative stability}, i.e.  whether (and how) the distance of a (classically defined) almost minimiser to the set of such infinitely concentrated  minimising configurations can be controlled in terms of the energy defect.

We note that the question of quantitative stability of minimisers for well known variational problems has attracted very significant attention in recent years, including in the context of the Isoperimetric problem, the Brunn Minkowski inequality, the  Yamabe problem, optimal Sobolev inequalities, spectral problems and conformal immersions, see e.g. \cite{sharp-isoperimetric,isoperimetric-FMP,BMI-1, Yamabe, Fusco-Sobolev,Figalli-Maggi-Pratelli-Sobolev, Sobolev-Maggi-Neumayer, Figalli-recent, Faber-Krahn, DeLellis-Mueller, Lamm-Nguyen,Luckhaus} and the references therein.

At the same time, there are few instances of quantitative stability results that apply in situations where the set of minimisers is genuinely non-compact (also modulo symmetries) or where the question of stability of minimisers is either false, or ill-posed, in the classical sense. 

A distinct feature of the problem under consideration, and a major challenge in the analysis, is that we want to compare objects of different topological type, namely a (classically defined) map $v$ from the given surface $\Si$ of positive genus with a minimising configuration that consists of a constant map from $\Si$ and an infinitely concentrated bubble. 

Before formulating any result, we hence first need to introduce a suitable notion of distance that captures the key features of the problem, i.e. that describes 
 \vspace{-0.5pc}
\begin{itemize}
    \item
   the distance of $v$ to the closest constant away from the concentration point; 
    \item 
    the distance of $v$ to the closest (rescaled) rotation on the bubble region; 
    \item
    the scale of concentration.
\end{itemize}
 \vspace{-0.5pc}
Formally, we can view a singular minimiser as a combination of a constant map $\om_0\equiv p\in S^2$ on $\Si\setminus \{a\}$ and an infinitely concentrated version $\RR\cdot \pi_\la$, $\la\to \infty$, of the map $\RR\cdot \pi: \R^2\to S^2$ that represents a given rotation $\RR\in SO(3)$ in stereographic coordinates. For the latter, it is natural to equip $\R^2$  with the metric $g_{\pi,\la}:=\pi_\la ^*g_{S^2}=\half \abs{\na \pi_\la}^2 dx^2$ that represents the standard metric of $S^2$ in the appropriately scaled stereographic coordinates.

To measure the distance of a given $v:\Si\to S^2$ from such a singular configuration, it is hence natural to consider 
 \vspace{-1pc}
\begin{enumerate}
\item the $\dot H^1(\Si\setminus B_{R}(a),g_\Si)$-norm, $R>0$, of $v-p$ and, as the behaviour of small regions does not contribute significantly to the $L^2$ norm,  the full $L^2(\Si,g_\Si)$ norm of $v-p$.
    \item the deviation from $\infty$ of the  scale $\la$ for which the $H^1(\DD_\iota, g_{\pi,\la})$ norm of $v_a-\RR\pi_\la$ is 
    minimal and the size of this norm.   
\end{enumerate}
 \vspace{-1pc}
Here, and in the following, we denote by $v_a$ the function that represents $v$ in standard conformal coordinates $x=F_a(p)$ around $a\in \Si$, i.e. coordinates which are chosen so that they provide an isometry from $B_I(a)$, $I=\inj(\Si,g_\Si)$,  to the disc $\DD_\iota=\DD_\iota(0)$ in $(\R^2,dx^2)$ for $\iota=I$ (if $\gamma(\Si)=1$) or in the Poincar\'e hyperbolic disc $(\DD_1,\frac{1}{(1-\abs{x}^2)^2}dx^2) $ for $\iota=\tanh(I)$  (if  $\gamma(\Si)\geq 2$). 
More generally, we denote by $B_R(a)$ geodesic balls in the given constant curvature domain $(\Si,g_\Si)$, and use the shorthands $r_\Si(R)$ for the radius of the corresponding disc, i.e. set $r_\Si(R)=R $ if $\gamma=1$ and $r_\Si(R)=\tanh(R)$ if $\gamma\geq 2$,  and $R_\Si(\cdot)$ for the inverse of this function $r_\Si(\cdot)$.

In this paper, we prove the following quantitative stability result 
that provides sharp bounds on all of the above key quantities that characterise the distance of a map to such infinitely concentrated minimisers.

\begin{thm}\label{thm:main}
Let $(\Sigma,g_\Si)$ be
any closed oriented surface with positive genus. Then there exists a constant $ C> 0$ such that the following holds true. \\
For any degree $1$ map $v\in H^1(\Sigma, S^2)$, there exists a rotation $\RR \in SO(3)$, a point $a\in \Sigma$, a value $p \in S^2$ and a scale $\lambda>1$ 
 such that the following holds. 
 \vspace{-1pc}
 \begin{enumerate}
 
\item  \label{part:thm-base}
The map $v$ is essentially given by the constant map $\om_0\equiv p$ on $\Si\setminus \{a\}$ in the sense that
\beq
\label{claim:constant-H01}
\|  \na v\|_{L^2(\Sigma \setminus B_{R}(a) ,g_\Si)} \leq C R^{-1} \delta_v^\half \text{ for every } R\in(0,\inj(\Si,g_\Si)), 
\eeq
and 
\beq
\label{claim:constant-L2}
\| v - p\|_{L^2(\Sigma ,g_\Si)} \leq C \delta_v^{\half}( 1 + \abs{\log \delta_v})^{\half}.
\eeq
     \item \label{part:thm-bubble}
    The function $v_a:= v\circ F_a^{-1}$ which describes $v$ in standard conformal coordinates around $a$
is essentially given by $\RR \pi_{\la}$ in the sense that 
\beq
\label{claim:thm-qs1}
\norm{v_a - \RR \pi_{\la}}_{H^1(\DD_{\iota},g_{\pi,\la}) }
\leq C \de_v^{\half},
\eeq 
$g_{\pi,\la}=\pi_{\la}^* g_{S^2}$, 
and the scale $\la$ satisfies 
\beq
\label{claim:la}
\la^{-1} \leq C\de_v^{\half}.
\eeq
\end{enumerate}
 \vspace{-1pc}
\end{thm}

We remark that the exponents in all of these estimates are sharp.
This is trivially true for the exponent $\half$ in  \eqref{claim:constant-H01} 
 and \eqref{claim:thm-qs1}
 as this is the maximal exponent for which a quantitative stability estimate of the form $\dist(u,\MM)\leqs \de_u^{\alpha}$ can hold for any smooth variational problem. 
 Similarly, 
 as cutting $\pi_\la$ off to $\pi(\infty)=(0,0,-1)$ on a fixed size annulus gives a map $v$ with $\de_v\sim\la^{-2}$ for which 
 \beq
 \norm{v-\fint_\Si v}_{L^2(\Si,g_\Si)} \geqs \norm{(\pi_\la^1,\pi_\la^2)}_{L^2(\DD_r)}\sim\la^{-1}(\log\la)^\half,
 \eeq
 we see that also the estimates  \eqref{claim:constant-L2} and \eqref{claim:la} 
are sharp.

\begin{rmk}
We note that the constant $C$ in the above theorem must be allowed to depend not only on the topology of $\Si$, but also on the metric $g_\Si$ (or some relevant geometric quantity associated to $g_\Si$ like a lower bound on $\inj(\Si,g_\Si)$ that excludes a degeneration in moduli space). To see this one can e.g. consider degree $1$ maps $v_j$ from degenerating hyperbolic surfaces $(\Si,g_j)$ with $E(v_j,g_j)\to 4\pi$ whose restriction to a disconnecting degenerating hyperbolic collar $(\mathcal{C}(\si_j), g_j) \simeq ([-X(\ell_j), X(\ell_j)] \times S^1, \rho_{\ell_j}^2 (s)(ds^2+d\theta^2))$, $\ell_j=L_{g_j}(\si_j)\to 0$, 
 converge to the harmonic map $(s,\theta)\mapsto \pi(e^{s+i\theta})$ from the infinite cylinder $\R\times S^1$. \\
 Similarly, we cannot expect that the above statements hold for $p=\RR\pi(\infty)$ as we can e.g. construct maps $v_\la:(\Si,g_\Si)\to S^2$ 
 which are given by $\pi_{\la}$ on a ball with radius $\la^{-\half}$ around some $b$, coincide with a constant $p\neq \pi(\infty)$ outsides of a fixed ball $B_{r}(b)$ and whose energy defect $\de_{v_\la}\sim (\log\la)^{-1}$ tends to zero as $\la\to \infty$.
\end{rmk}

We note that in the case where $\Si$ is a sphere, the question of quantitative stability of degree $1$ 
minimisers is significantly different, and far less delicate, than for domains of positive genus. Indeed, if $\Si=S^2$ then the minimum is not only achieved, but the set of minimisers indeed agrees with the symmetry group $\textit{M\"ob}(S^2)$ of M\"obius transforms. In practice this means that the difficulty posed by concentration of energy can be avoided by working in the correct gauge, compare \cite{Topping-deg-1}, and quantitative stability of minimisers holds in the classical sense. Indeed,  Bernand-Mantel, Muratov and Simon \cite{BMS} first established that
\beq
\label{est:QS-S2-deg1}
\dist_{\dot H^1(S^2,g)}(v,\text{M\"ob}(S^2))\leq C \de_v^{\half}
\eeq
holds true 
for every degree $1$ map $v:S^2\to S^2$ and a universal constant $C$, alternative proofs of this result
were subsequently given by Hirsch and Zemas in \cite{H-Z} using elliptic regularisation techniques and by Topping in \cite{Topping-deg-1} using the harmonic map flow, and extensions to higher dimensions were recently proven by  Guerra, Lamy and Zemas in \cite{ZLG1,ZLG2}.
We note that a quantitative understanding of the Dirichlet energy of maps into $S^2$ is not only of interest in geometric analysis, but also in applications where the Dirichlet energy (also called exchange energy) forms a key term e.g. in models from material science, see e.g. \cite{DIL,MSkyr} and the references therein.

We also note that the quantitative stability of minimisers for maps $v:S^2\to S^2$ of arbitrary degree $k\geq 2$ was addressed by the first author in \cite{R-Rig-23}. In this case the set of (classical) minimisers $\MM_k(S^2)$ is given by the degree $k$ rational functions, so non-empty, but stability of minimisers fails in the classical sense even at the qualitative level, i.e. there exist sequences of maps $v_n$ of any given degree $k\geq 2$ with $E(v_n)\to 4\pi k$ for which $\dist_{H^1}(v_n,\MM_k(S^2))\nrightarrow 0$. However, if one expands the class of comparison objects to additionally include collections of rational functions of total degree $k$ which are concentrated on essentially disjoint subsets then not only qualitative, but also quantitative stability of minimisers can be recovered. Indeed, the results of \cite{R-Rig-23} provide sharp estimates 
of the form $\dist\leq C \de_u^\half (1+\abs{\log \de_u})^\half$ on the distance of any given degree $k$ map to the nearest collection of such rational maps, and provide additional (not necessarily sharp) control on the relative scales and location of the set on which the maps that make up this collection are concentrating, see Theorems 1.1 and Theorems 1.4 of \cite{R-Rig-23} for detail. As the arguments in \cite{R-Rig-23} exploit the symmetry group  $\textit{M\"ob}(S^2)$  of the domain $\Si=S^2$ and the fact that all components of the limiting configuration can be parametrised over this $\Si$ they are not applicable in the present setting of maps from surfaces of positive genus. We hence need to develop a new and more flexible approach that provides a well controlled way to deform a map from a given surface into a limiting map that is defined on a domain of different topology.

\textit{Overview of the key steps in the proof of our main result:}\\
A key aspect of our proof is that we will simultaneously deform the map and the domain. 
To achieve the latter we will construct families of metrics that deform suitably chosen representatives of the given conformal class $[g_\Si]$ of metrics on the domain surface of positive genus into a punctured sphere in a well controlled way that in particular allows us to keep the metric essentially fixed on the bubble region.

To this end we first consider metrics $\bar g_{\mu,b}$, $b\in \Si$, $\mu\geq \mu_*=\mu_*(g_\Si)$,  which are conformal to $g_\Si$ but weighted in a way that 
they scale the ball $B_{R_\Si(\mu^{-1})}(b)$ up to a hemisphere, are so that a fixed size ball  $B_{R_2}(b)$ in the original surface becomes isometric to $S^2\setminus B^{g_{S^2}}_{s}(\pi(\infty))$ for some $s\sim \mu^{-1}$ and so that $g_\Si$ gets scaled down by a constant $c_{\mu}\sim \mu^{-2}$ outside of a slightly larger ball $B_{R_0}(b)$. In particular, in the limit $\mu\to \infty$ these weighted metrics turn $\Si$ into a punctured sphere.  

Instead of directly working with these metrics $\bar g_{\mu,b}$, we then construct a second family of metrics $g_{\mu,b}$ which are isometric to $\bar g_{\mu,b}$ but chosen in a way that the curves $\mu\mapsto g_{\mu,b}$ provides a far more efficient way of deforming $(\Si,g_{\mu,b})$ into a punctured sphere than what would be possible by just conformal deformations. 
To be more precise, given any $b\in \Si$
we will construct diffeomorphisms $T_{\mu,b}:\Si\to \Si$, $\mu\geq \mu_*$,  so that the resulting metrics 
$g_{\mu,b}=T_{\mu,b}^*\bar g_{\mu,b}$ satisfy 
$\norm{\partial_\mu g_{\mu,b}}_{L^2(\Si,g_{\mu,b})}\sim \mu^{-2}$, allowing us to deform a given such metric $g_{\mu_0,b}$ into a singular metric $g_{\infty,b}$ by moving only a distance of
\beq \label{est:dist-metric-to-infty}
\int_{\mu_0}^\infty \norm{\partial_\mu g_{\mu,b}}_{L^2(\Si,g_{\mu,b})}\sim \mu_0^{-1}.\eeq
We will also be able to ensure that these metrics are essentially ordered in the sense that the inequality $g_{\mu,b}\geq c g_{\tilde \mu,b}$ holds for all $\tilde \mu\geq \mu$ and a universal constant $c>0$ (later chosen as $c=e^{-3})$.

Given a map $v$ with small energy defect 
we will then first choose $b\in \Si$ and $\mu_0\gg 1$ 
so that 
\beq
\label{eqn:exact-bubble}
E(v,B_{R_\Si(\mu_0^{-1})}(b))=\sup_{\tilde b \in \Sigma} E(v,B_{R_\Si(\mu_0^{-1})}(\tilde b))=2\pi. 
\eeq
We then devise a way that allows us to deform the initial pair $(U_0:= v,G_0:=\bar g_{\mu_0,b})$ of map and metric 
into a limit $(U_\infty,G_\infty)$ which consists of
 \vspace{-1pc}
\begin{itemize}
    \item a singular metric $G_\infty$ which is supported on a subset $ \Om_\infty$ of $\Si$ that corresponds to a ball $B_{\hat R_\infty}(b)$ with  $\hat R_\infty\in (R_2,R_0)$ in the original surface $(\Si,g_\Si)$,  agrees with  $G_0$ on a fixed set $B_{R_2}(b)$ and is so that $(\Om_\infty,G_\infty)$ is isometric to the punctured sphere $S^2\setminus \{\pi(\infty)\}$, 
    \item a harmonic map $U_\infty:(\Om_\infty,G_\infty)\to S^2$. 
\end{itemize}
 \vspace{-1pc}
Crucially, we will be able to carry out this deformation through a curve of maps  $U(t)$ and metrics $G(t)=(T_{\mu_0,b})_*g_{\mu(t),b}$ for which   
\beq \label{est:key-dist-intro}
\int_0^\infty \norm{\pt U}_{L^2(\Si,G(t))} +\norm{\pt G}_{L^2(\Si,G(t))} dt\leq C \de_v^{\half} .\eeq
This will yield both the desired bound \eqref{claim:la} on the bubble scale, compare \eqref{est:dist-metric-to-infty}, as well as a weighted $L^2$ estimate for the map component of the form 
$$\norm{v-U_\infty}_{L^2(\Om_\infty,G_\infty)}\leq C \de_v^{\half}$$
from which we will then be able to derive all other claims  made in our main result.

To obtain this curve $(U(t),G(t))$ we first 
 pull-back both the given map $U_0=v$ and the given weighted metric $G_0=\bar g_{\mu_0,b}$ by the corresponding (fixed) diffeomorphism $T_{\mu_0,b}: \Si \to \Si $ to obtain a new pair $(u_0,g_0)=T_{\mu_0,b}^*(U_0,G_0)=(v\circ T_{\mu_0,b},g_{\mu_0,b}) $
that represents the given configuration in new coordinates, but is now so that $g_0$ is an element of our specially constructed family of metrics $\GG_{b}:=\{g_{\mu,b}, \, \mu\geq \mu_*\}$.

We then deform $(u_0,g_0)$ using a new geometric flow that evolves both the map and the metric, similarly to \cite{R-T-TMHF-1}. Namely we consider the negative $L^2$-gradient flow of the Dirichlet energy $\mathcal{E}_b=E\vert_{H^1(\Si,S^2)\times \GG_b}$, considered as a function of both an (arbitrary) $H^1$ map $u:\Si\to S^2$ and a metric $g$ that is constraint to lie in this $1$-parameter family of metrics $\GG_b$ for the fixed point $b$ chosen in \eqref{eqn:exact-bubble}. Formally,  this flow is given by 
\beq
\label{def:flow-abstract}
\pt u=-\na_u^{L^2(\Sigma,g)}E(u,g), \qquad \pt g=-P^{T_g\GG_b}( \na_g^{L^2(\Sigma,g)}E(u,g)),
\eeq
where $P^{T_g \GG_{b}}$ denotes the $L^2(\Si,g)$ orthogonal projection from the space of $(0,2)$-tensors to $T_g \GG_{b}$. 

We recall that the negative $L^2$-gradient of the Dirichlet energy with respect to the  map component is given by the tension field, so here where we map into $S^2$ by 
$$  - \nabla_u^{L^2(\Si,g)} E(u,g) = \tau_g(u) =\Delta_g u + \abs{du}_g^2 u.$$
We also recall that the negative $L^2$-gradient of the Dirichlet energy with respect to general variations of the metric component is given in terms of the energy stress-tensor, namely 
\beq \label{def:energy-stress}
- \nabla_g^{L^2(\Si,g)} E(u,g) = \tfrac12k(u,g) \text{ for }
k(u,g):= u^* g_{S^2} - \tfrac{1}{2}\abs{du}_g^2 g.\eeq 
Our flow hence evolves a given initial map $u_0\in H^1(\Si,S^2)$ and a given initial metric $g_0\in \GG_b$ by
\beqa
\label{eq:flow}
\pt u = \tau_g(u), \qquad
\pt g = \tfrac{1}{2}P^{T_g\mathcal{G}_b}\left(k(u,g)\right).
\eeqa
A crucial role in the analysis of this flow 
is played by a \Loj estimate of the form 
\beq \label{est:Loj-intro}
\abs{E(u,g)-4\pi}\leq C \norm{\na \EE_b(u,g)}_{L^2(\Si,g)}^2, 
\eeq
which 
provides uniform bounds on the energy defect in terms of the norm of the gradient 
$\na \EE_b(u,g)=\big(\na_u^{L^2(\Si,g)} E (u,g) \,,\, P^{T_g\mathcal{G}_b} \na_g^{L^2(\Si,g)} E(u,g)\big)
$ of the restricted energy functional $\EE_b=E\vert_{H^1(\Si,S^2)\times \GG_b}$.

We prove two versions of 
 this and related \Loj estimates that might be of independent interest
 in Section \ref{sec:Loj-Est}, one in the more typical form of a result that applies in a suitable neighbourhood of the set of critical points (which in our case are however infinitely concentrated), see Theorem \ref{thm:Loj-NHD}, and a version for maps with small energy defect that is crucially used in the proof of our main result, see 
  Theorem \ref{thm:Loj-new}. As explained in Section \ref{sec:Loj-Est}, this second version can be obtained from the first version based on a careful analysis of almost critical points of $ \EE_b$, which in particular uses that control on $P^{T_g\mathcal{G}_b} \na^{L^2(\Si,g)}_g E$ excludes the formation of non-trivial necks in the image even when the geometry of the domain degenerates, compare Remark \ref{rmk:excluding-necks}. Theorem \ref{thm:Loj-NHD} in turn will be proven in 
Sections \ref{sec:proof-LojEst} and \ref{sec:technical-proofs}
based on the approach developed by Malchiodi, Sharp and the first author in \cite{MRS}.

For maps with small energy defect this key \Loj estimate will allow us to a priori exclude the formation of singularities of the flow that are caused by the bubbling off of harmonic spheres,  hence allowing us to continue the flow smoothly until it achieves its goal of changing the topology of the domain from a surface of positive genus to a punctured sphere. 
Equally importantly, the \Loj estimate \eqref{est:Loj-intro} yields a priori bounds on the distance travelled by solutions of \eqref{eq:flow}, and hence by $(U,G)(t)=(T_{\mu_0,b})_* (u,g)(t)$,
in the precise form  \eqref{est:key-dist-intro} that we require to derive the sharp quantitative stability estimates stated in our main result.

\textit{Outline of the paper:} \newline
The 1-parameter families $\GG_b$ of metrics that allow for the efficient and well controlled change of the topology are constructed in Section \ref{sec:metrics} and their key properties are summarised in Proposition \ref{prop:metrics-main}.  We then state the relevant \Loj estimates in Section \ref{sec:Loj-Est} and use them to analyse the behaviour of solutions of our new gradient flow in Section \ref{sec:flow}.
This will allow us to complete the proof of our main result in Section 
 \ref{sec:thm1-proof}, up to the proof of our \Loj estimates which is carried out in Sections \ref{sec:proof-LojEst} and \ref{sec:technical-proofs}.

 \textit{Notation and conventions:}
\\
Throughout the paper we use the following conventions: Unless indicated otherwise constants $C$ are positive and only depend on the fixed  domain $(\Si,g_\Si)$, with 
 the shorthands $\leqs$, $\sim$ and $\geqs$ being used if inequalities hold up to such multiplicative constants $C=C(\Si,g_\Si)>0$, quantities on $\Si$ are computed with respect to $g_\Si$,  
 quantities on discs $\DD_r$ are computed with respect to the standard metric $dx^2$ and derivatives with respect to $\eps$ and $\de$ are evaluated at $\ep=0=\de$.
\\
 As noted above, we will always use local conformal coordinates $x=F_a(p)$, $a\in \Si$, which provide isometries from $(B_R(a),g_\Si)$, $R\in (0,I=\inj(\Si,g_\Si)]$  to $(\DD_{r_\Si(R)},\rho_\Si^2(x)dx^2)$, where $ \rho_\Si\equiv 1$ and $r_\Si(R)=R\in (0,\iota:=I)$ if $\gamma=1$ while $\rho_\Si(x)=(1-\abs{x}^2)^{-1}$ and $r_{\Si}(R)=\text{tanh}(R)\in (0,\iota:= \tanh(I)]$ if $\gamma\geq 2$, further normalised by Remark \ref{rmk:coord-normal} if we work with varying basepoints. We denote by $F_{a,b}:= F_{b}\circ F_a^{-1}$ the corresponding transition maps, ordering indices so that we transition from the first named chart to the second one. 

\textit{Acknowledgement:} SW was supported by the Engineering and Physical Sciences Research
Council.

\section{Definition and basic properties of metrics}\label{sec:metrics}
In this section we provide a detailed construction and analysis of the metrics which we use to prove our main result. 
To define these metrics $\bar g_{\mu,b}$ and $g_{\mu,b}$, $b\in \Si$,  $\mu\geq \mu_*=\mu_*(\Si,g_\Si)>1$, we will work with balls $B_{R_i}(b)$ in $(\Si,g_\Si)$ with fixed radii $R_i=R_{\Si}(r_i)=r_{\Si}^{-1}(r_i)$  for 
\beq\label{def:ri-mu-star}
r_0:= \mfrac12 \min(\iota,1),\qquad 
r_1:= \mfrac14 r_0, \qquad r_2:= e^{-4} r_1, \quad r_3=e^{-1} r_2, \text{ and set } 
\mu_*:= e^5  r_1^{-1}=r_3^{-1}.
\eeq 
Fixing some $\phi \in C_c^{\infty}(\DD_{r_0},[0,1])$ with $ \phi \equiv 1 \text{ on } \DD_{\frac{r_0}{2}}$, we
define $\bar{g}_{\mu,b}=\bar\rho_{\mu,b}^2g_\Si$,  
 $b \in \Si$,  $\mu \geq \mu_*$,
by 
 \beq
\label{def:conformal-factor-constant}
\bar{g}_{\mu,b} =c_{\mu}^2 g_{\Si} \text{ on }  \Si\setminus  B_{R_0}(b) \text{ for }
c_\mu := \tfrac1{\sqrt{2}} \abs{ \na \pi_{\mu}(r_0e^{i\theta})}=
\tfrac{2\mu}{1+\mu^2r_0^2}
\eeq
and by $\bar g_{\mu,b}=\phi\circ  F_b \cdot F_b^*g_{\pi,\mu}+(1-\phi\circ F_b)\cdot c_\mu^2 g_\Si$
on $ B_{R_0}(b)$, i.e. chose $\bar{g}_{\mu,b}$ so that it is 
represented by 
\beq
\label{def:conformal-factor-vary} 
\phi \ g_{\pi,\mu}
+(1-\phi)c_\mu^2
 \rho_{\Si}^2 dx^2, \qquad g_{\pi,\mu}=\pi_\mu^*g_{S^2}
\eeq
in the usual local conformal coordinates $F_b:B_{R_0}(b)\to \DD_{r_0}$. 

We note that these  weighted metrics $\bar g_{\mu,b}$ are equivalent to those used in \cite{R-Loj-1} and are defined in a way that 
sets which are given by geodesic balls $B_{R_\Si(r)}(b)$, $0<r\leq \half r_0$, in the original surface turn into sets  $(B_{R_\Si(r)}(b),\bar g_{\mu,b})$ which are isometric to the subset $\pi(\DD_{\mu r})$ of $S^2$.
In particular for large $\mu$ the ball $B_{R_\Si(\mu^{-1})}(b)$ of radius $R_\Si(\mu^{-1})\sim \mu^{-1}\ll 1$ gets scaled up to a hemisphere, while 
for any fixed $R\in (0, R_{\Si}(\half r_0)]$ the sets $(B_R(b),\bar g_{\mu,b})$ are isometric to subsets that exhaust the punctured sphere $S^2\setminus\{\pi(\infty)\}$ as $\mu\to \infty$.

Conversely, these metrics scale down subsets of $\Si\setminus \{b\}$ and we note that bounds of the form 
 \beq\label{est:pmu-bulk}
\bar\rho_{\mu,b}\sim \mu^{-1} \text{ and } \abs{\partial_\mu \bar\rho_{\mu,b}}\sim \mu^{-2} \text{ on } \Si\setminus B_{c}(b)
\eeq
are valid for any fixed  $c>0$ and immediately follow as 
\beq 
\label{est:simple-pi-estimates}
\abs{\na \pi_\mu} 
\sim \tfrac{1}{\mu \abs{x}^2} \quad \text{ while }\quad \pmu(\tfrac{1}{\sqrt{2}} \abs{\na \pi_{\mu}(x)}) = 2 \tfrac{1-\mu^2\abs{x}^2}{(1+\mu^2\abs{x}^2)^2}\sim -\tfrac{1}{\mu^2 \abs{x}^2} \quad \text{ for } \abs{x}\geq 2\mu^{-1}.\eeq

As $r_1=\frac14r_0\leq \half r_0$ we have 
$\phi\equiv 1$ on $\DD_{r_1}$ and hence obtain 

\begin{rmk}
\label{rmk:weighted-metric-new}
On $B_{R_1}(b)$ the metric 
$\bar g_{\mu,b}$ is given $g_{\pi,\mu}=\pi_\mu^*g_{S^2}$ in the usual conformal coordinates, so can be obtained from the fixed metric $g_\pi$ by pulling-back with the dilation $D_\mu: x\mapsto \mu x$. 
At times it will hence be convenient to work in the rescaled cylindrical coordinates $(s=\bar s_{\mu,b},\theta)$, $\bar s_{\mu,b}(p)=\log(\mu \abs{F_b(p)})$ in which $\bar g_{\mu,b}$ is given by 
the metric 
\beq 
\label{eq:stand-cyl-metric} 
 \rho^2 \gcyl \text{ for } \gcyl:= d s^2+d\theta^2  \text{ and  }
\rho(\cdot):=\cosh(\cdot)^{-1}
\eeq
that represents $g_{\pi}$ in standard cylindrical coordinates.
\end{rmk}

Instead of directly working with these metrics, we pull them back by suitably chosen diffeomorphisms 
$T_{\mu,b}: \Sigma \to \Sigma$, thus associating to each point $b\in \Si$ a $1$-parameter family of metrics 
$$\GG_b:=\{g_{\mu,b} := T_{\mu,b}^{*}\bar{g}_{\mu,b}, \quad \mu\geq  \mu_*\}.$$
To make this precise, we consider the family of diffeomorphisms 
$T_\mu: \DD_{r_0} \to \DD_{r_0}$, 
$\mu \geq \mu_*$,
which is generated by 
\beqa
\label{def:Tmu}
\ppeps T_{\mu,\mu+\eps}(x) 
=- \mu^{-1} \phi_\GG(x) x \text{ from } T_{\mu_*}(x)=x, \qquad  T_{\mu,\tilde \mu}:=T_{\tilde \mu}\circ T_{\mu}^{-1},
\eeqa
where $\phi_\GG$ is chosen as $\phi_\GG(x):= 1-\psi(\log(r_1^{-1}\abs{x}))$ for a function $\psi\in C^\infty(\R)$ with 
\beqa
\label{def:psi-g}
 \psi\equiv 0 \text{ on } (-\infty,-4],  \quad \psi(-2)=\mfrac 12,  \quad \psi\equiv 1 \text{ on } [0,\infty)\text{ and } 0\leq \psi'\leq \mfrac13.
\eeqa
This ensures that 
$\phi_\GG(x)\equiv 0$ for $\abs{x}\geq r_1=e^4r_2$ and $\phi_\GG\equiv 1$ on $\DD_{r_2}$ and hence that 
\beq 
\label{ass:Tmu} 
T_\mu =\text{ Id} \text{ on } \DD_{r_0} \setminus \DD_{r_1} \text{ and }
T_\mu(x)= \mu^{-1} \mu_* x \text{ on } \DD_{r_2}.\eeq
The rotational symmetry of $\phi_\GG$ allows us to write $T_\mu(r e^{i\theta})=t_\mu(r) e^{i\theta}$ for functions $t_\mu$ which, by \eqref{def:Tmu}, are non-increasing in $\mu$ 
and we furthermore note that 
\beq \label{eq:mu-t-agree}
\mu T_\mu\equiv \mu_1 T_{\mu_1} \text{ on } \DD_{r_\GG(\mu_1)} \text{ for } 
\mu\geq \mu_1, \quad  \text{ where } r_\GG(\mu_1):= t_{\mu_1}^{-1}(r_2)\eeq
since  $\pmu(\mu t_\mu(r))=\psi(\log(t_\mu(r) r_1^{-1}))t_\mu(r)=\psi(\log(t_\mu(r) r_2^{-1})-4)t_\mu(r)\equiv 0$ whenever $t_\mu(r)\leq r_2$.

\begin{rmk}\label{rmk:nice-coordinates}
$S_\mu(x):=\mu T_\mu(x)=\mu t_\mu(x) e^{i\theta}$  represents an isometry from $(B_{R_1}(b),g_{\mu,b}:= T_{\mu,b}^* \bar g_{\mu,b})$ to $(\DD_{\mu r_1},g_\pi)$ in the usual conformal coordinates 
since $T_{\mu,b}$ maps $B_{R_1}(b)$ onto itself and since $\bar g_{\mu,b}$ is described by $g_{\pi,\mu}=D_\mu^*g_\pi$ on this ball. 
Denoting by $f_\mu(s):=\log(\mu t_\mu(e^s))$ the function which represents $S_\mu$ in cylindrical coordinates, i.e.  which is so that 
$e^{f_\mu(s)}= \mu t_\mu(e^s)$,
 hence yields
coordinates 
\newcommand{\cyl}{\text{cyl}}
$ 
(s_{\mu,b}(p):= f_\mu(\log(\abs{F_b(p)}) ,\theta(p)) $  on $B_{R_1}(b)$ in which $g_{\mu,b} $ is given by $\rho^2g_{\cyl}$, compare \eqref{eq:stand-cyl-metric}. We note that these coordinates correspond to the coordinates $(\bar s_{\mu,b},\theta)$ in the pulled-back viewpoint.  
\end{rmk}
\begin{rmk}\label{rmk:switch-viewpoint}
We often use the diffeomorphisms $T_{\mu,b}$ to 
switch viewpoint, noting that as long as we only consider fixed values of $\mu$ and apply the same diffeomorphism to all involved maps and metrics, such a change of viewpoint does not change any geometrically defined quantities. 
Conversely, quantities that involve variations of $\mu$ such as norms of $\partial_\mu g_{\mu,b}$ and $\pmu \bar g_{\mu,b}$ can scale very differently and indeed the key point of our construction of $\GG_b$ is to obtain a far more efficient way of changing the metrics than what would be possible using conformal changes. 
\end{rmk}

We now show that our particular choice of $T_\mu$ ensures that $\GG_b$ has the following key properties. 

\begin{prop}
\label{prop:metrics-main}
Let $T_{\mu,b}: \Sigma \to \Sigma$ be the family of diffeomorphisms that is represented by the functions $T_\mu$ defined in \eqref{def:Tmu} above in the usual local coordinates on $B_{R_0}(b)$ and  
given by $T_{\mu,b} = Id \text{ on } \Sigma \setminus B_{R_0}(b)$. 
Then the induced family of metrics
\beq
\label{def:G_b}
\GG_b:=\{g_{\mu,b} := T_{\mu,b}^{*}\bar{g}_{\mu,b}, \quad \mu\geq  \mu_*\}, \qquad  \bar g_{\mu,b} \text{ defined by \eqref{def:conformal-factor-constant} and \eqref{def:conformal-factor-vary} } \eeq
 has the following properties. 
 \vspace{-1pc}
\begin{enumerate}
    \item \label{item:metric-for-free}
The metrics agree on balls of increasing radius 
$R_\GG(\mu_1):=
R_{\Si}(t_{\mu_1}^{-1}(r_2))\in (R_2,R_1)$  in $(\Si,g_\Si)$,
\beq \label{claim:metrics-identical}
g_{\mu,b}\equiv g_{\mu_1,b} \text{ on }
B_{R_\GG(\mu_1)}(b) = T_{\mu_1,b}^{-1}(B_{R_2}(b))\text{ for } \mu\geq \mu_1,
 \eeq 
and the resulting limiting metric $$g_{\infty,b}:=\lim_{\mu\to \infty} g_{\mu,b} \text{ on } B_{R_\infty}(b), \quad R_\infty:= \lim_{\mu\to \infty} R_\GG(\mu)$$ is so that 
$( B_{R_\infty}(b),g_{\infty,b})$ is isometric to the punctured sphere
$ (\R^2, g_{\pi}) \cong (S^2\setminus \pi(\infty),g_{S^2})$ via the  diffeomorphism that is represented by $ \lim_{\mu \to \infty} \mu T_\mu:\DD_{r_\infty}\to \R^2$, $r_\infty=r_{\Si}(R_\infty)=\lim_{\mu\to\infty} t_\mu^{-1}(r_2)$,
 in the usual conformal coordinates around $b$, compare \eqref{eq:mu-t-agree}.
\item
\label{metric:item-estimates}
We have bounds of
\beq
\label{est:Met-Var}
\| \pmu g_{\mu,b} \|_{L^2(\Sigma,g_{\mu,b})} \sim \mu^{-2}
\quad \text{ and }\quad \| \pmu g_{\mu,b} \|_{L^{\infty}(\Si,g_{\mu,b})} \sim \mu^{-1}.
\eeq

\item\label{item:metric-almost-monotone}
The metrics are essentially ordered in the sense that  
\beq
\label{eq:metric-almost-monotone}
g_{\mu,b} \leq e^3 g_{ \mu_1,b} \text{ holds  on all of } \Si \text{ and  for all } \mu\geq \mu_1.
\eeq

\item \label{item:metric-v} 
The trace-free part of 
 $\pmu g_{\mu,b}$ is supported on a subset of  $B_{R_1}(b)\setminus B_{R_2}(b)$ and can be written as 
\beq \label{eq:cylinder-tracefree} \mathring{\pmu g_{\mu,b}}=\mu^{-1}\psi'(s_{\mu,b} -X_\mu)\rh^2(s_{\mu,b})  (ds_{\mu,b}^2-d\theta^2)\eeq
in the coordinates $(s_{\mu,b},\theta)\in (-\infty,X_\mu]\times S^1$, $X_\mu:=\log(\mu r_1)$
on $B_{R_1}(b)$ in which $g_{\mu,b}$ is given by 
$\rho^2 g_{\text{cyl}} $, compare Remark \ref{rmk:nice-coordinates}.
\end{enumerate}

\end{prop}
We note that the detailed information on the trace-free part of $\pmu g_{\mu,b}$ is helpful since it allows us to obtain the following explicit expression for the norm of the velocity 
$\pt g=\frac12 P^{T_g\mathcal{G}_b}\left(k(u,g)\right)$ of our flow.
\begin{cor}\label{cor:proj-cylinder}
Let $g=g_{\mu,b} \in \GG_b$ and let $\tilde u$ be the map that represents a given map 
$u \in H^1(\Si,S^2)$ in the coordinates $(s_{\mu,b},\theta)$ from Remark \ref{rmk:nice-coordinates}. 
Then 
\beq
\label{est:proj-cylinder}
\| P^{T_{g} \GG_b}(k(u,g)) \|_{L^2(\Sigma, g)} = c_{b}(\mu)  \left|\int 
\psi'(s-X_\mu) \al(s)ds \right|  \text{ for } \al(\cdot):= \int_{\{\cdot\}\times S^1} \abs{\tilde u_s}^2 - \abs{\tilde u_\theta}^2 d\theta  
\eeq
and $c_b(\mu):=\mu^{-1}\norm{\pmu g_{\mu,b}}_{L^2(\Si,g_{\mu,b})}^{-1}\sim \mu$.
In particular, we can always bound  
\beq \label{est:trivial-norms-gt} 
\| P^{T_{g} \GG_b}(k(u,g)) \|_{L^2(\Sigma, g)} \leqs E(u,g) \mu \quad\text{ while }\quad \| P^{T_{g} \GG_b}(k(u,g)) \|_{L^\infty(\Sigma, g)} \leqs E(u,g) \mu^2.
\eeq
\end{cor}

\begin{proof}[Proof of Proposition \ref{prop:metrics-main}]
As 
$S_\mu(x)=\mu T_\mu(x)$  represents an isometry from $(B_{R_1}(b),g_{\mu,b})$ to $(\DD_{\mu r_1},g_\pi)$ the first statement immediately follows from \eqref{eq:mu-t-agree}.

It is also easy to see that $g_{\mu,b}$
has the required properties on 
 $\Si_0:=\Sigma \setminus B_{R_{\Si}(r_1)}(b)$ where it agrees with 
$\bar g_{\mu,b}=\bar \rho_{\mu,b}^2 g_\Si$. Indeed, the (almost) monotonicity required in \eqref{eq:metric-almost-monotone} trivially holds as
$\mu\mapsto \bar \rho_{\mu,b}$ is non-increasing
and \eqref{est:pmu-bulk} ensures that 
 $\abs{\pmu g_{\mu,b}}_{g_{\mu,b}} = 2 \sqrt{2} (\bar\rho_{\mu,b})^{-1} \abs{\pmu \bar\rho_{\mu,b}}\sim \mu^{-1}$ on $\Si_0$ and that $\Area(\Si_0,\bar g_{\mu,b})\sim \mu^{-2}$ which hence yields
$ \| \pmu g_{\mu,b} \|_{L^2(\Si_0
,g_{\mu,b})}\sim \mu^{-2}.$ 

For the rest of the proof we can hence focus on the analysis on $B_{R_1}(b)$ and for \ref{metric:item-estimates} on deriving upper bounds on the norms of $\pmu g_{\mu,b}$.
We will carry out the proofs of \ref{metric:item-estimates} and \ref{item:metric-v} in the coordinates 
$ (s_\mu=s_{\mu,b},\theta)$ in which we can write \beqs
g_{\tilde \mu,b}= (\rho\circ f_{\mu,\tilde \mu})^2([(
f_{\mu,\tilde \mu}')^2 ds_\mu^2 + d\theta^2] \quad \text{ for any $\tilde \mu$.}
\eeqs
 As $e^{f_\mu(s)}= \mu t_\mu(e^s)$ we can write the corresponding transition maps
$f_{\mu,\tilde \mu}:=f_{\tilde \mu}\circ f_\mu^{-1}$ as 
$$ f_{\mu,\tilde \mu}(s)=\log(\tilde \mu t_{\mu,\tilde \mu}(\mu^{-1} e^s))=\log(\tilde \mu)+ \log( t_{\mu,\tilde \mu}(\mu^{-1} e^s)).$$
As 
$\peps \log( t_{\mu, \mu+\eps}(\mu^{-1} e^s))=\mu e^{-s}\peps t_{\mu,\mu+\eps}(\mu^{-1} e^s)=-\mu^{-1} \phi_\GG(\mu^{-1}e^s)$
this implies that the functions $f_\mu$ are generated by 
$$\peps f_{\mu,\mu+\eps}(s)=Y_\mu(s) :=\mu^{-1} \psi(s-X_\mu) \text{ for } X_\mu:= \log (\mu r_1). $$
As $\rho'=-\tanh \cdot \rho$ we hence obtain that 
\beqa
\label{def:pmu-gmu-cylindrical}
 \pmu g_{\mu,b}&=2\rho(s_\mu)  (\rho'(s_\mu)Y_\mu(s_\mu) + \rho(s_\mu) 
Y_\mu'(s_\mu))ds_\mu^2 + 2\rho(s_\mu)\rho'(s_\mu)Y_\mu(s_\mu) d\theta^2  \\
&=2 (Y_\mu'(s_\mu)-\tanh(s_\mu)\cdot Y_\mu(s_\mu))\rho^2(s_\mu) ds_\mu^2-
2\tanh(s_\mu)Y_\mu(s_\mu) \rho^2(s_\mu) d\theta^2\\
&=2 \mu^{-1}[\psi'(s_\mu-X_\mu)-\tanh(s_\mu)\cdot \psi(s_\mu-X_\mu)]\rho^2(s_\mu) ds_\mu^2\\
&\qquad -
2\mu^{-1} \tanh(s_\mu)\psi(s_\mu-X_\mu)
\rho^2(s_\mu) d\theta^2.
\eeqa 
From this we can read off the formula for $\mathring{\pmu g_{\mu,b}}$ claimed in \ref{item:metric-v} as well as that $\norm{\pmu g_{\mu,b}}_{L^\infty(B_{r_1}(b), g_{\mu,b})}
\leqs \mu^{-1}$. 
Since $\pmu g_{\mu,b}\equiv 0$ on $B_{{R_\GG}(\mu)}(b)$, which corresponds to points with $s_\mu\leq f_\mu(\log(r_{\GG}(\mu))=\log(\mu r_2)=X_\mu-4$, we have  
$\Area_{g_{\mu,b}}(\supp(\pmu g_{\mu,b})\cap B_{R_1}(b))\leq 2\pi \int_{X_\mu-4}^{\infty} \rho^2(s_\mu)ds_\mu \leqs \int_{\log(\mu r_2)}^\infty e^{-2s_\mu}ds_\mu\leqs \mu^{-2}$,
allowing us to bound
$ \norm{\pmu g_{\mu,b}}_{L^2(B_{R_1}(b), g_{\mu,b})}\leqs \mu^{-2}$. This completes the proof of part \ref{metric:item-estimates}.

It hence remains to establish \ref{item:metric-almost-monotone} at points $p$ which are in $B_{R_1}(b)$ and for this we work in the standard  cylindrical coordinates $(s=\log(\abs{F_b(p)},\theta)$. 
Fixing any such point $p$ and the corresponding $s\in (-\infty, \log(r_1)]$ we can now use that \eqref{def:pmu-gmu-cylindrical} implies that 
\beq
\label{est:metric-pmu-1}
\partial_\mu g_{\mu,b}(p)   \leq v_\mu(s)g_{\mu,b}(p) 
\text{ for } 
 v_\mu(s) =2\mu^{-1} \big[ \psi'(f_\mu(s)-X_\mu)-\tanh(f_\mu(s))\psi(f_\mu(s)-X_\mu)\big].
 \eeq
We of course have $v_\mu(s)=0$ for all $\mu$ for which  $f_\mu(s)\leq X_\mu-4$  for the given $s$ since $\psi$ is supported on $[-4,\infty)$. 
For all other $\mu$ we can use that $X_\mu-4\geq X_{\mu_*}-4=\log(\mu_*r_2)=1$ to bound  $\tanh(f_\mu(s))\geq \tanh(1)>\frac34$. As $0\leq \psi'\leq \frac13$, this ensures that  
\beq \label{est:vmu}
v_\mu(s)\leq 2\mu^{-1} \big[\mfrac13 -\mfrac34\psi(f_\mu(s)-X_\mu)\big] \quad \text{for all $\mu$ with } f_\mu(s)-X_\mu\geq -4.
\eeq
As $\psi\geq\half$ on $[-2,\infty)$ this quantity is negative whenever $f_\mu(s)\geq X_\mu-2$, so we obtain that 
$$\pmu g_{\mu,b}(p)\leq 0 \text{ for all $\mu$ with } f_\mu(s)-X_\mu \notin (-4,-2).$$
Hence \ref{item:metric-almost-monotone} follows once we show that 
$$g_{\mu_2,b}(p)\leq e^3g_{\mu_1,b}(p) \text{ for all } \mu_1\leq \mu_2 \text{ for which } f_{\mu_i}(s)-X_{\mu_i}\in (-4,-2)\text{ for } i=1,2.$$
Since
$\pmu(X_\mu-f_{\mu})= \mu^{-1}-Y_\mu\circ f_\mu\geq 0$ we note that 
 $f_{\mu}(s)-X_\mu\in (-4,-2)$ indeed holds on the whole interval $ [\mu_1,\mu_2]$. This allows us to bound 
$Y_\mu(f_\mu(s))\leq \half \mu^{-1}$  and hence  $\pmu(X_\mu-f_{\mu}(s))\geq \half \mu^{-1}$ on all of $[\mu_1,\mu_2]$. 
Combined with \eqref{est:vmu}
this ensures that 
$$v_\mu(s)\leq \mfrac23\mu^{-1}\leq \mfrac{4}{3} \pmu(X_\mu-f_{\mu}(s))\text{ for all } \mu\in [\mu_1,\mu_2]\text{ and hence } \int_{\mu_1}^{\mu_2} v_\mu(s) d\mu \leq \mfrac{8}3\leq 3.$$ 
This immediately gives the desired estimate $g_{\mu_2}(p)\leq e^3  g_{\mu_1}(p)$ as \eqref{est:metric-pmu-1} ensures that  
$$\pmu \log(g_{\mu,b}(p)(w,w)) \leq v_\mu(s(p)) \quad \text{ for every } \quad  w\in T_p\Si. $$
\end{proof}
\begin{proof}[Proof of Corollary \ref{cor:proj-cylinder}]
Since 
$\GG_b$ is one-dimensional and since $k(u,g)$ is trace-free while $\supp(\mathring{\pmu g_{\mu,b}})\subset B_{R_1}(b)\setminus B_{R_2}(b)$ we have 
\beqs
\| P^{T_{g} \GG_b}(k(u,g)) \|_{L^2(\Sigma, g)} = 
\norm{\pmu g_{\mu,b}}_{L^2(\Si,g)}^{-1} 
\abs{\langle \mathring{\partial_\mu g_{\mu,b}}, k(u, g) \rangle_{L^2(B_{R_1}(b)\setminus B_{R_2}(b), g)}}.
\eeqs
The expression for $\mathring{\partial_\mu g_{\mu,b}}$ in the coordinates $(s=s_{\mu,b},\theta)$ given in \eqref{eq:cylinder-tracefree} and the fact that the $ds^2-d\theta^2$ part of $k(\tilde u,\rho^2\gcyl)$ is given by 
$\frac12( \abs{\tilde u_s}^2-\abs{\tilde u_\theta}^2)(ds^2-d\theta^2)$ immediately gives the claimed formula \eqref{est:proj-cylinder}. 
The first bound in \eqref{est:trivial-norms-gt} then immediately follows  while the second holds since $\norm{k(u,g)}_{L^1(\Si,g)}\leqs E(u,g)$ and since we can bound
$\norm{P^{T_{g} \GG_b}(k)}_{L^\infty(\Si,g)}\leq \norm{k}_{L^1(\Si,g)}\frac{\norm{\pmu g}_{L^\infty(\Si,g)}^2}{\norm{\pmu g}_{L^2(\Si,g)}^2}\leqs \mu^2  \norm{k}_{L^1(\Si,g)}$ for any tensor $k$.
\end{proof}

\section{\Loj Estimates for
maps with small energy defect}\label{sec:Loj-Est}

Before we state our main \Loj estimates, we  collect some basic properties of maps which will be used throughout this and the following sections. 

To begin with, we recall that harmonic maps $\om:\R^2\to S^2$ with $\deg(\om)=\pm 1$ can be represented by $x\mapsto \RR\pi(\la J(x-x_0))$ in stereographic coordinates, for $\RR\in SO(3)$, $x_0\in \R^2$ and $\la > 0$, where $J$ is chosen as the identity if $\deg(\om)=1$, respectively a reflection if $\deg(\om)=-1$.
We also remark that $x_0$ and $\la > 0$ can be uniquely characterised by the condition that 
\beq
\label{fact:rotation}
E(\om,\DD_{\la^{-1}}(x_0))=\sup_{y\in \DD_{\la^{-1}}} E(\om,\DD_{\la^{-1}}(y))=2\pi
\eeq
a fact which can be easily checked, e.g. by using the explicit expression for
 $\abs{\na \pi(x)} = \frac{2\sqrt{2}}{1+\abs{x}^2}$ to see that 
$\partial_{\eta}\vert_{\eta=0} E(\pi,\DD_r((1+\eta)x))<0$ for any $r>0$ and $x\neq 0$.

We will also repeatedly use the following remark on the behaviour of maps with bounded tension, which is an immediate consequence of the 
compactness theory of almost harmonic maps developed in \cite{Ding-Tian, Lin-Wang,Qian-Tian,Struwe}.

\newcommand{\loc}{\text{loc}}
\begin{rmk}
\label{rmk:basic-compactness}
Let $(M,g)$ be any closed surface and let $v_n\in H^2(M,S^2)$ be any sequence of maps with 
$$\sup_{n} E(v_n)<8\pi \text{ and } \sup_{n} \norm{\tau_g(v_n)}_{L^2(M,g)}<\infty.$$
   Then, after passing to a subsequence, one of the following statements holds. \vspace{-1pc}
   \begin{itemize}
       \item $v_n$ converges weakly in $H^2_{\loc}(M,g)$, and hence in particular strongly in $H^1_\loc\cap L^\infty_{\loc}$, to a limit $v_\infty$.
       \item $v_n$ converges weakly in $H^2_{\loc}(M\setminus \{b\},g)$ to a limit $v_\infty$ away from a point $b\in \Si$ at which a single degree $\pm 1$ bubble forms; that is, there exist $\la_n\to \infty$, $x_n\to 0$, $\RR\in SO(3)$ and $J$ as above so that the maps $v_{n,b}$ which represent $v_n$ in the usual conformal coordinate chart $F_b:B_{I}(b)\to \DD_{\iota}$ are so that 
       $$v_{n,b}(x)-(\RR\pi(\la_n J(x-x_n)) - \RR \pi(\infty))-v_{\infty,b}\to 0 \text{ strongly in } H^1\cap L^\infty (\DD_{\iota}).$$
   \end{itemize}
\vspace{-1pc}
If $\norm{\tau_g(v_n)}_{L^2(M,g)}\to 0$ then the above statements hold with strong, rather than weak, $H^2$ convergence and the limiting map $v_\infty$ will be harmonic. 
\end{rmk}

\begin{rmk}
\label{rmk:no-SED-harmonic}
Since the Riemann-Hurtwiz theorem precludes the existence of a degree $1$ map from $\Si$ to $S^2$ with energy $4\pi$, it follows from \cite[Theorem 1.1 and Corollary 1.3]{R-Loj-1} that there exists $\de_* =\de_*(g_\Si)> 0$ such $E(u) \geq 4\pi + \de_*$ holds for all degree $1$ harmonic maps $u: \Si \to S^2$.
\end{rmk}

These remarks in particular ensure that for degree $1$  maps $u_n:\Si\to S^2$ whose energy defect tends to zero and whose tension remains bounded, all of the energy
will concentrate near a single point. 
Indeed, this statement remains valid also for maps with unbounded tension, namely we have
\begin{lemma}\label{lemma:SED-concentrate}
For all $R \in (0,I)$ and all $c>0$, there exists $\de_1 > 0$ such that 
\beq
\label{est:SED-concentrate}
\sup_{b\in\Si}  E(v,B_{R}(b))  \geq 4\pi - c,
\eeq
for all degree $1$ maps $v \in H^1(\Si,S^2)$ with $\de_v \leq \de_1$.
\end{lemma}
For the proof of this lemma we recall that the energy decays according to $\ddt E(u)=-\norm{\pt u}_{L^2(\Si,g_\Si)}^2$
along solutions $u(t)$ of the classical harmonic map flow $\pt u=\tau_{g_\Si}(u)$ on the \textit{fixed} surface $(\Si,g_\Si)$, and that local energy estimates of the form  
\beqa \label{est:loc-energy-conc}
E(u(t_2),B_{\frac{R}{2}}(b)) &\leq E(u(t_1),B_{R}(b)) + CR^{-1}  \int_{t_1}^{t_2}\norm{\pt u}_{L^2(\Si,g_\Si)} dt\\
&\leq E(u(t_1),B_{R}(b))+ C R^{-1} (E(t_1)-E(t_2))^{\half}(t_2-t_1)^{\half}
\eeqa are valid for a constant $C$ that only depends on an upper bound $\bar E$ on the initial energy.

\begin{proof}[Proof of Lemma \ref{lemma:SED-concentrate}]
As observed above, we can choose $\tilde \de_1>0$ so that the stronger bound of 
\beq
\label{est:SED-concentrate-1}
\sup_{b\in\Si}  E(u,B_{R/2}(b), g_\Si)  \geq 4\pi - \tfrac{c}{2},
\eeq holds 
for all degree $1$ maps $u$ with 
 $\de_u\leq \tilde \de_1$ which additionally satisfy 
$\| \tau_{g_\Si}(u) \|_{L^2(\Si,g_\Si)} \leq 1$.
We now set 
$\de_1:=\min(\tilde \de_1,\frac{c R}{2C},\pi,\de_*)$, $\de_*>0$ as in Remark \ref{rmk:no-SED-harmonic} and $C$ so that  \eqref{est:loc-energy-conc} holds for $\bar E=5\pi$.

Given $v$ as in the lemma, we now let $u(t)$ be the corresponding solution of the harmonic map flow and let $T_{max}\in (0,\infty]$ be the maximal time until which the flow exists and remains smooth.  
Since $v$ is not homotopic to a harmonic map, we know that $u(t)$ must become singular as $t\upto T_{\max}$, irrespective of whether or not $T_{\max}$ is finite. Therefore,  \eqref{est:SED-concentrate-1} will hold eventually, allowing us to consider the 
first time $T_1\in[0,T_{\max})$ for which it is valid. 

If $T_1=0$ then the claim trivially holds. Otherwise, we can use that 
 the definition of $ \tilde \de_1$ and the fact that  $\de_{u(t)}\leq \de_v\leq  \de_1\leq \tilde \de_1$ ensure that 
 $ \| \tau_{g_\Si}(u(t)) \|_{L^2(\Si,g_\Si)}>1$ at all times where 
 \eqref{est:SED-concentrate-1} is violated, so in particular on all of $[0,T_1)$. This allows us to bound
$$T_1 \leq \int_0^{T_1} \| \tau_{g_\Si}(u) \|_{L^2(\Si,g_\Si)}^2 dt =E(v)-E(u(T_1))=\de_v-\de_{u(T_1)}\leq \de_1,
$$ 
which, when inserted into \eqref{est:loc-energy-conc}, 
allows us to bound 
\beqas 
4\pi-\tfrac{c}{2}\leq \sup_{b\in\Si}  E(u(T_1),B_{R/2}(b)) & \leq \sup_{b \in \Si} E(v,B_{R}(b)) + \tfrac{C}{R} \de_1 \leq \sup_{b \in \Si} E(u,B_{R}(b))+\tfrac{c}{2},
\eeqas
and hence deduce the claimed estimate \eqref{est:SED-concentrate}.
\end{proof}
\begin{rmk}
\label{rmk:SED-concentrate-bubble}
We can in particular apply this lemma for $c = 2\pi$ to deduce that for any degree $1$ map $v$ with small energy defect there exist $\mu_0\geq 4\mu_*$ and $b\in\Si$ so that \eqref{eqn:exact-bubble} is satisfied. 
\end{rmk}
While we will use this specific choice of parameters later to associate to $v$ an initial metric $G_0=\bar g_{\mu_0,b}$ for our flow, for the analysis of (degree $1$) maps $\bar u$ with small energy defect it suffices to ensure that we are always working with metrics $\bar g_{\mu,b}$ which are compatible with the energy distribution of $\bar u$ in the sense of \eqref{ass:Loj-thm} below. We note that unlike \eqref{eqn:exact-bubble}, this weaker condition \eqref{ass:Loj-thm} will remain valid along our flow, hence allowing us to control the evolution of the flow using the following key \Loj estimate.

\begin{thm}\label{thm:Loj-new}
For every closed orientable surface $(\Si,g_\Si)$ of positive genus, there exist constants $\delta_1 > 0$, $\eps_{2}>0 $ and $C$ so that the following holds true.

Let $\bar u \in H^1(\Si,S^2)$ be any degree $1$ map with $\de_{\bar u}  \leq \delta_1$ and let $\mu\geq \mu_*$ and $b\in \Si$ be so that 
\beq
\label{ass:Loj-thm}
E(\bar{u},B_{R_{\Si}(\frac{1+\eps_2}{\mu})}(b))
\geq 2\pi - \eps_2 \text{ and } \max_{d_{g_\Si}(\tilde b,b)\leq R_{\Si}(\mu^{-1}) }E(\bar{u},B_{R_{\Si}(\frac{1-\eps_2}{\mu})}(\tilde b)) \leq 2\pi + \eps_2.
\eeq
Then, setting  $ u:= T_{\mu,b}^*\bar u$, we can bound 
\beq
\label{est:LE-energy}
\abs{E(\bar{u}) - 4\pi} \leq C \norm{\na \EE_b(u,g_{\mu,b})}_{L^2(\Si,g_{\mu,b})}^2 
\eeq and 
\beq\label{est:LE-mu}
\mu^{-1}\leq C  \norm{\na \EE_b(u,g_{\mu,b})}_{L^2(\Si,g_{\mu,b})},
\eeq
and there exist $\la \sim \mu$, $a\in \Si$ and $\RR\in SO(3)$ so that $\bar u_a:= \bar u\circ F_a^{-1}$ satisfies 
\beqa
\label{est:dist-bound-1}
 \| \bar{u}_a - \RR \pi_{ \la}\|_{H^1(\DD_{\iota}, \pi_{\la}^*g_{S^2})}
 +  \| \bar{u} - \RR \pi(\infty)\|_{H^1(\Si\setminus B_{I}(a), c_{\la}^2 g_{\Si})}\leq C \norm{\na \EE_b(u,g_{\mu,b})}_{L^2(\Si,g_{\mu,b})}.
\eeqa
\end{thm}
Here $T_{\mu,b}, \bar g_{\mu,b}$ and $g_{\mu,b}=T_{\mu.b}^*\bar g_{\mu,b}\in \GG_b$ are the diffeomorphisms and metrics constructed in Section \ref{sec:metrics} and $\EE_b$ denotes the restriction of the Dirichlet energy onto $H^1\times \GG_b$. The norm appearing on the right hand side of the above inequalities can hence be equivalently computed as
\beq \label{def:EG} 
\norm{\na \EE_b(u,g_{\mu,b})}_{L^2(\Si,g_{\mu,b})}^2
= \| \tau_{\bar{g}_{\mu,b}}(\bar{u}) \|_{L^2(\Sigma,\bar{g}_{\mu,b})}^2 + \tfrac{1}{4}\left\| \bar P_{\bar{g}_{\mu,b}}(k(\bar{u},\bar g_{\mu,b})) \right\|_{L^2(\Sigma,\bar g_{\mu,b})}^2,
\eeq
$\bar P_{\bar{g}_{\mu,b}}$ the $L^2$-orthogonal projection onto the tangent space of $(T_{\mu,b}^{-1})^* \mathcal{G}_b$ at $\bar g_{\mu,b}=(T_{\mu,b}^{-1})^* g_{\mu,b}$.

At times it will be useful to translate the condition \eqref{ass:Loj-thm} into a condition on the energy of the map $\bar u_{b}$ on suitable discs, and for this it is useful
to note
\begin{rmk}
\label{rmk:ball-disc}
For $\tilde b$ close to $b$, the set $F_b(B_{R}(\tilde b))$, which represents 
 $B_R(\tilde b)$, $R\leq  I-d_{b,\tilde b}$, in the conformal coordinates centred at $b$,
is well approximated by $\DD_r(F_b(\tilde b))$, $r:= r_{\Si}(R)$, in the sense that 
 \beq 
\DD_{(1-\eps)r}(F_b(\tilde b))\subset F_b(B_{R}(\tilde b))\subset \DD_{(1+\eps)r}(F_b(\tilde b)) \text{ holds for some } \eps\sim  d_{b,\tilde b}^2+
d_{b,\tilde b} R,
 \eeq
 $d_{b,\tilde b}:=d_{g_\Si}(b,\tilde b)$, 
compare \eqref{est:err-bound}.
\end{rmk}
In the specific situation of the theorem, the number $\eps$ scales like $O(\mu^{-2})$ which, for maps with sufficiently small energy defect, will in particular be less than $\half \eps_2$, compare Lemma \ref{lemma:SED-concentrate}. We hence obtain
\begin{rmk} \label{rmk:checking-E-coord}
Theorem \ref{thm:Loj-new} is applicable for all maps $\bar u$ with sufficiently small energy defect for which 
\beq \label{ass:check-in-coord-bar}
E(\bar u_b,\DD_{(1+ \eps_2)\mu^{-1}})\geq 2\pi-\eps_2 \quad \text{ and } \quad \sup_{\abs{y}\leq \mu^{-1}} E(\bar u_b,\DD_{(1-\half \eps_2)\mu^{-1}}(y)) \leq 2\pi+\eps_2.
\eeq
\end{rmk}
We will prove Theorem \ref{thm:Loj-new} based on the following result, which is established in Section \ref{sec:proof-LojEst}.
\begin{thm}\label{thm:Loj-NHD}
For every closed orientable surface $(\Si,g_\Si)$ of positive genus there exist constants $\eps_1>0$, $\lambda_0>1$ and $C$ such that the following holds true.

Let $\bar{u} \in H^1(\Sigma,S^2)$ be a degree $1$ map which is $H^1\cap L^\infty$ close to a simple bubble tree in the sense that there exist $\mu\geq \la_0$, $b \in \Si$ and  $\RR_0\in SO(3)$ so that $\bar u_b:=\bar u\circ F_b^{-1}$ satisfies 
\beq
\label{ass:LE-map-1}
 \|\na(\bar u_b - \RR_0 \pi_{\mu}) \|_{L^2(\DD_{\iota})}  + \|\bar u_b - \RR_0 \pi_{\mu}\|_{L^\infty(\DD_{\iota})}\leq \eps_1  
\eeq
and
\beq
\label{ass:LE-map-2}
 \|\na \bar{u} \|_{L^2(\Si\setminus B_{I}(b))}  + \| \bar{u} - \RR_0 \pi(\infty) \|_{L^\infty(\Si\setminus B_{I}(b))}\leq \eps_1.
\eeq
Then, the \Loj estimates \eqref{est:LE-energy}, \eqref{est:LE-mu} and \eqref{est:dist-bound-1}
hold true for some 
 $a\in \Si$, $\la\sim \mu$ and $\RR\in SO(3)$.  
 \end{thm}

\begin{rmk} \label{rmk:old-Loj}
We note that a \Loj estimate of the form 
\beq
\label{est:Loj-old}
\abs{E(\bar u) - 4\pi} \leq C \| \tau_{g_\Si}(\bar u) \|_{L^2(\Sigma,g_\Si)}^2\left(1+\abs{\log \| \tau_{g_\Si}(\bar u) \|_{L^2(\Sigma,g_\Si)}}\right),
\eeq
that involves the tension with respect to the standard metric was obtained by the first author in \cite{R-Loj-1} and that \Loj estimates for maps from $S^2$ to itself were obtained by Topping \cite{Topping-rigidity, Topping-HMF-2} and Waldron \cite{Alex-Loj}. 
While \eqref{est:Loj-old} and \eqref{est:LE-energy} might at first glance look quite similar, the way in which they weigh the tension in different regions of the domain is indeed very different as $\| \tau_{\bar \rho_{\mu,b}^2 g_\Si}(\bar u)\|_{L^2(\Si,\bar \rho_{\mu,b}^2 g_\Si)} = \| \bar \rho_{\mu,b}^{-1} \tau_{g_\Si}(\bar u) \|_{L^2(\Si,g_\Si)}$ where $\bar \rho_{\mu,b}\gg 1$ near $b$ while $\bar \rho_{\mu,b}\ll 1$ away from $b$.
\end{rmk}

As the above \Loj estimates are all trivially true if 
$\norm{\na \EE_b(u,g_{\mu,b})}_{L^2(\Si,g_{\mu,b})}$ is bounded away from zero, we can  derive 
Theorem \ref{thm:Loj-new} from Theorem \ref{thm:Loj-NHD} using the following lemma.

\begin{lemma}
\label{lemma:derive-thm}
Let $\bar u_n \in H^1(\Si,S^2)$, $\deg(\bar u_n)=1$, and $\bar g_n=\bar g_{\mu_n,b_n}$ be so that
\beq
\label{est:prop-comp-almost}
\tau_n:= \| \tau_{\bar{g}_n}(\bar{u}_n) \|_{L^2(\Sigma,\bar{g}_{n})}\to 0, \qquad \mathcal{P}_n:=\left\| \bar P_{\bar{g}_n}(k(\bar{u}_n,{\bar{g}_n})) \right\|_{L^2(\Sigma,{\bar{g}_n})}\to 0 \text{ and } \de_{\bar{u}_n}\to 0 
\eeq
and so that 
there exists  $\eps_n \to 0$ with
\beq
\label{est:prop-comp-ELB}
E(\bar{u}_n,B_{R_{\Si}(\frac{1+\eps_n}{\mu_n})}(b_n))
\geq 2\pi - \eps_n \text{ and } \max_{d_{g_\Si}(\tilde b,b_n)\leq R_{\Si}(\mu_n^{-1}) }E(\bar{u}_n,B_{R_{\Si}(\frac{1-\eps_n}{\mu_n})}(\tilde b)) \leq 2\pi + \eps_n.
\eeq
Then, there exists $\RR_n \in SO(3)$ such that $\bar u_{n,b_n}:= \bar u_n\circ F_{b_n}^{-1}$ satisfies
\beq
\label{ass:LE-lemma-H1}
 \|\na( \bar u_{n,b_n} - \RR_n \pi_{\mu_n}) \|_{L^2(\DD_{\iota})}  +
\|\na \bar u_n \|_{L^2(\Si\setminus B_{I}(b_n))}  \to 0
 \eeq
and
\beq
\label{ass:LE-lemma-Linfty}
  \| \bar u_{n,b_n} - \RR_n \pi_{\mu_n}\|_{L^\infty(\DD_{\iota})}
+ \| \bar u_n -  \RR_n \pi(\infty) \|_{L^\infty(\Si\setminus B_{I}(b_n))}\to 0 .
\eeq
\end{lemma}

 \begin{rmk}\label{rmk:excluding-necks}
 As the metrics 
 $\bar g_{\mu_n,b_n}$ degenerate, this lemma result does not directly follow from the theory of almost harmonic maps from fixed surfaces developed in \cite{Ding-Tian,Lin-Wang,Qian-Tian,Struwe}, though we will be able to exploit these results to see that 
 \eqref{ass:LE-lemma-H1} holds 
even without the assumption that $\mathcal{P}_n\to 0$. 
Conversely, this assumption is crucial  to exclude the formation of necks on the degenerating part of the surfaces $(\Si,\bar g_{\mu_n,b_n})$ and hence in the proof of \eqref{ass:LE-lemma-Linfty}. 
  \end{rmk}

\begin{proof}[Proof of Lemma  \ref{lemma:derive-thm}]
As it suffices to 
 show the lemma for subsequences, we can assume without loss of generality that $b_n\to b_\infty\in \Si$. We also recall that $\de_{\bar u_n}\to 0$ implies that $\mu_n\to \infty$. 
 
 To establish the lemma we will prove that there exists $\RR\in SO(3)$ so that, after passing to a subsequence,
\begin{eqnarray}
\label{claim:hallo-1}
\lim_{n\to\infty}  \|\na(  \bar u_{n,b_n} - \RR \pi_{\mu_n}) \|_{L^2(\DD_{\La \mu_n^{-1}})}  + \|\bar u_{n,b_n} - \RR \pi_{\mu_n}\|_{L^\infty(\DD_{\La \mu_n^{-1}})}&=& 0 \text{ for every } \La<\infty,\\
\label{claim:hallo-2}
\lim_{n \to \infty} \osc_{\Si\setminus B_R(b_n)} \bar u_n&=& 0 \text{ for every }R>0, \\
\label{claim:hallo-3}
\lim_{\La \to \infty} \limsup_{n\to \infty} 
\osc_{A_n(\La,c)}  \bar u_{n,b_n} &=& 0  \text{ for every  } c \in (0, \iota),
\end{eqnarray}
for $A_n(\La,c):=\DD_{c}\setminus \DD_{\La \mu_n^{-1}}$.
As $\osc_{A_n(\La,c)} \RR \pi_{\mu_n}=
\osc_{\DD_{c\mu_n}\setminus \DD_{\La }
} \pi\leq 
\osc_{\R^2\setminus \DD_{\La } } \pi\to 0$ as $\La \to \infty$ these estimates immediately imply \eqref{ass:LE-lemma-Linfty}.

As $E(\RR \pi_{\mu_n},\DD_{\La \mu_n}) =E( \pi,\DD_{\La })\to 4\pi$ as $\La\to \infty$ and  as $E(\bar u_n)\to 4\pi$ we also immediately obtain  from \eqref{claim:hallo-1} that both 
\begin{eqnarray}
\label{est:hallo-E-bulk}
\lim_{n\to\infty}\norm{\na \bar u_n}_{L^2(\Si\setminus B_R(b_n))}&=&0 \text{ for every } R>0, \text{ and }  
\\
\label{claim:hallo-3-E}
\lim_{\La \to \infty}\limsup_{n\to \infty} \norm{\na \bar u_{n,b_n}}_{L^2(A_n(\La,c))} &=& 0  \text{ for every } c\in (0,\iota) .
\end{eqnarray}
 Thus \eqref{ass:LE-lemma-H1} immediately follows  from \eqref{claim:hallo-1} and the fact that 
$\norm{\na (\RR \pi_{\mu_n})}_{L^2(A_n(\La,c))}
\leq \norm{\na \pi}_{L^2(\R^2\setminus \DD_{\La })}\to 0$ as $\La\to \infty$. This reduces the proof of the lemma to the proof of \eqref{claim:hallo-1}-\eqref{claim:hallo-3}.

To prove  \eqref{claim:hallo-1} we consider the 
rescaled maps 
$\tilde u_n(x):= \bar u_{n,b_n}(\mu_n^{-1} x)$
which are so that
$$\norm{\tau_{g_{\pi}}(\tilde u_n)}_{L^2(\DD_{R_n},g_\pi)}
=\norm{\tau_{g_{\pi,\mu_n}}(\bar u_{n,b_n})}_{L^2(\DD_{\half r_0},g_{\pi,\mu_n})}\leq \tau_n\to 0$$ for radii
$R_n:=\frac{r_0}{2} \mu_n\to \infty$ and so that 
$$E(\tilde u_n,\DD_{1+\eps_n})\geq 2\pi-\eps_n \quad \text{ and }\quad \sup_{\abs{y}\leq 1} E(\tilde u_n,\DD_{1-\eps_n-C\mu_n^{-2}}(y))\leq 2\pi+\eps_n, $$
compare \eqref{est:prop-comp-ELB} and Remark \ref{rmk:ball-disc}. As  $E(\bar u_n)=4\pi+\de_n<5\pi$ these maps cannot form any bubbles and must thus 
subconverge strongly in $H^2_{\loc}(\R^2)$ to a limiting harmonic map $u_\infty$ with $E(u_\infty)\leq 4\pi$ and 
\beq\label{eq:energy-uinfty}
E(u_\infty,\DD_1)=\sup_{\abs{y}\leq 1} E(u_\infty,\DD_1(y))=
2\pi.
\eeq
This harmonic limit must have degree 1, as a negative degree is excluded since the maps $\bar u_n$ have positive degree and small energy defect while $\deg(u_\infty)=0$ is excluded since $u_\infty$ is not constant. 
This allows us to write $u_\infty=\RR \pi(x)$ for some $\RR\in SO(3)$, compare \eqref{fact:rotation}, and deduce that \eqref{claim:hallo-1} holds. 

To establish \eqref{claim:hallo-2} for a given $R>0$
we use that  
the metrics $\bar g_n$ are all conformal to $g_\Si$ and that their conformal factor scales like $\mu_n^{-1}$ outside of balls of fixed radius around $b_n$. As $\mu_n\to \infty$ and $b_n\to b_\infty$ we can hence in particular bound  $\rho_{\mu_n,b_n}\leq 1$ on $\Si\setminus B_{R/4}(b_n)\supset \Si\setminus B_{R/2}(b_\infty)$ for all sufficiently large $n$ and we hence find that
$\norm{\tau_{g_\Si}(\bar u_n)}_{L^2(\Si\setminus B_{R/2}(b_\infty),g_\Si)}\leq \tau_n\to 0$. 
As the already established \eqref{claim:hallo-1} ensures that $E(\bar u_n, \Si\setminus B_{R/2}(b_\infty))\to 0$, compare \eqref{est:hallo-E-bulk}, the maps $\bar u_n$ must thus subconverge in $ H^2(\Si\setminus B_{3R/4}(b_\infty))$ to a constant map which in particular establishes \eqref{claim:hallo-2}. 

To establish the remaining claim
\eqref{claim:hallo-3},
we work in the rescaled cylindrical coordinates $(s=\bar s_{\mu_n,b_n},\theta)$ introduced in Remark \ref{rmk:weighted-metric-new}, in which $\bar g_n$ is given by $\rho^2 g_{\text{cyl}}$ and let $v_n$ be the maps which represent $\bar u_n$ in these coordinates. 
We note that points $x\in A_n(\La,c)$ have coordinates $(s=\log(\mu_n\abs{x}),\theta)\in I_n(\La,c)\times S^1$ for $I_n(\La,c)=[\log(\La),\log (c\mu_n)]$ 
and that the claim \eqref{claim:hallo-3} follows once we show that both
\beq 
\label{claim:hallo-4}
\sup_{s \in I_n(\La,c)
} \osc_{S^1\times \{s\}} v_n =o_{\La,n}
\text{ and }
\int_{I_n(\La,c)\times S^1}\abs{\partial_{s} v_n} d s d\theta =o_{\La,n},
\eeq
where we write for short $o_{\La,n}$ for quantities with $\lim_{\La\to \infty} \limsup_{n \to \infty} o_{\La,n}=0$.

As \eqref{claim:hallo-3-E} ensures that 
$E(v_n, I_n(e\La,e^{-1}c)\times S^1)= 
E(\bar u_{n,b_n}, A_n(e\La,e^{-1}c))=o_{\La,n}$, we know in particular that the energy on 
 cylinders of the form $[s-1,s+1]\times S^1$, $s\in I_n(\La,c)$, will be smaller than any given constant if $n$ and $\La$ are sufficiently large. This allows us to apply 
 standard angular 
energy estimates as found e.g. in  \cite[Lemma $2.1$]{H-R-T-TMF} 
to bound 
 $
\Theta_n(s) := \int_{\lbrace s \rbrace \times S^1} \abs{\partial_\theta v_n}^2 d\theta$  by
\beqa
\label{est:theta-s}
\Theta_n(s) & \leqs E(v_n, I_n(e\La,ce^{-1}) \times S^1) e^{-\dist(s,\partial I_n(e\La,e^{-1}c))} + \int_{I_n(e\La,ce^{-1})\times S^1} \abs{\tau_{g_{\text{cyl}}}(v_n)}^2 e^{-\abs{s-q}}d\theta dq\\
&= o_{\La,n} e^{-\dist(s,\partial I_n(\La,c))}+
\int_{I_n(e\La,ce^{-1})\times S^1} \abs{\tau_{\rho^2 \gcyl}(v_n)}^2 \rho^2(q)e^{-\abs{s-q}} dv_{\rho^2 g_{\text{cyl}}}\\
&= o_{\La,n} e^{-\dist(s,\partial I_n(\La,c))}+\tau_n e^{-\abs{s}}
\eeqa
for all $s\in I_n(\La,c)$ where the last step follows as $\rho(q)\sim e^{-\abs{q}}$.

This not only implies that 
$\sup_{I_n(\La,c)} \Theta_n^\half=o_{\La,n}$ and hence that the first claim in \eqref{claim:hallo-4} holds, but also that 
$\int_{I_n(\La,c)} \Theta_n^\half ds=o_{\La,n}$ and hence that 
\beqa
\int_{I_n(\La,c)\times S^1}\abs{\ps  v_n}& \leqs \int_{I_n(\La,c)} \big(\int_{S^1} \abs{\ps v_n}^2\big)^{\half} =\int_{I_n(\La,c)}  (\al_n+\Theta_n
)^{\half}\leqs o_{\La,n}+\int_{I_n(\La,c)}  \abs{\al_n}^\half 
\eeqa
for $ \al_n(s):= \int_{\{s\}\times S^1}  \abs{\ps v_n}^2 - \abs{\partial_\theta v_n}^2 d\theta$. It thus remains to show that also $\int_{I_n(\La,c)}  \abs{\al_n(s)}^\half ds=o_{\La,n}$.

We recall that the anti-holomorphic derivative of the function $\phi= \abs{v_s}^2 - \abs{v_\theta}^2 - 2 i v_s \cdot v_\theta$, that describes the 
Hopf-differential $\phi dz^2$, is given by 
$(\ps + i \partial_\theta) \phi= 2(v_{ss} + v_{\theta \theta})(v_s - i v_\theta)=2\tau_\gcyl(v)(v_s-i v_\th)$. As $\int_{S^1} \partial_\theta \phi=0$ we can thus write 
$\partial_s \alpha_n(s) = \int_{\{s\}\times S^1}\text{Re}((\ps + i \partial_\theta) \phi_n ) d\theta
=2\int_{S^1} \tau_{\gcyl}(v_n) \cdot \partial_s v_n d\theta$ and estimate
\beqa
\label{est:alpha-distance}
\abs{\alpha_n(s_1) - \alpha_n(s_2)}
& \leqs E(v_n)^\half \norm{\tau_{\gcyl}(v_n)}_{L^2([s_1,s_2]\times S^1)} \leqs  e^{-s_1} \tau_n
\text{ for all } s_1<s_2\leq X_{\mu_n}= \log(r_1 \mu_n).
\eeqa
In particular, $\abs{\alpha_n(s) - \alpha_n( X_{\mu_n})}\leqs  e^{-(X_{\mu_n}-4)} \tau_n\leqs \mu_n^{-1} \tau_n$ on $\supp(\psi'(\cdot-X_{\mu_n}))\subset [X_{\mu_n}-4,X_{\mu_n}]$ for $\psi$ as in \eqref{def:psi-g}.  
We now crucially use that the assumption that $\PP_n\to 0$ implies that 
$$\mu_n\left| \int 
\psi'(s-X_{\mu_n})\al_n(s) ds \right| \sim \PP_n\to 0,$$ 
see Corollary \ref{cor:proj-cylinder} and Remark \ref{rmk:switch-viewpoint}.  
As $\psi'\geq 0$ and 
$\int \psi'(s)ds = 1$ we can hence bound 
$$\mu_n \abs{\al_n(X_{\mu_n})}\leq  \mu_n \left|\int 
\psi'(s-X_{\mu_n})\cdot [\al_n(X_{\mu_n})-\al_n(s)] ds\right| +C\PP_n\leqs 
\tau_n+\PP_n.$$
Inserted into 
 \eqref{est:alpha-distance} this gives   
$\abs{\alpha_n(s)} \leqs (e^{-s}+\mu_n^{-1})o_{n}$ 
where $o_n\to 0$ as $n\to \infty$. As $\abs{I_n(\La,c)}\leqs \log \mu_n$ we hence obtain that indeed $\int_{I_{n}(\La,c)}\abs{\alpha_n(s)}^\half ds \leq (1+\mu_n^{-\half}\log \mu_n) o_n=o_n$ for every $\La$ and $c$, thus completing the proof. 
\end{proof}

\section{Analysis of our new gradient flow}\label{sec:flow}
Since  $\GG_b$ is 1-dimensional, we can view \eqref{eq:flow} as a system that consists of a semi-linear parabolic PDE coupled with a first-order ODE with locally Lipschitz right hand side. Short time existence of solutions for initial data $(u_0,g_0)\in C^{2,\al}(\Si,S^2)\times \GG_b$  
can hence be obtained by a standard iteration argument, see e.g. \cite[Appendix A]{R-TMF-2}. 
The definition of the flow \eqref{eq:flow} ensures that the energy decays according to 
\beqa
\label{eq:ddt-E}
\tfrac{d}{dt} E(u,g) 
= - \| \partial_t (u,g) \|_{L^2(\Sigma,g)}^2=-\norm{\na \EE_b(u,g_{\mu,b})}_{L^2(\Si,g_{\mu,b})}^2
\eeqa
along smooth solutions $(u(t),g(t)=g_{\mu(t),b})$ while the evolution of localised energies 
\beqs
E_{\phi}(u,g) := \frac{1}{2} \int_{\Sigma} \phi^2 \abs{du}_{g}^2 dv_{g},
\eeqs
obtained using \textit{time independent} cut-off functions 
 can be controlled by the following standard estimates.

\begin{lemma}\label{lem:loc-energy}
For any $\phi \in C^{\infty}(\Sigma,[0,1])$ and any smooth solution  $(u,g)$ of \eqref{eq:flow}
we have 
\beqa
\label{est:loc-energy-1}
\left| \mfrac{d}{dt} E_{\phi}(u,g) + \int_{\Sigma} \phi^2 \abs{\partial_t u}^2 dv_g \right| \leq C \| \partial_t g \|_{L^{\infty}( \text{supp}(\phi),g)} + C \| \nabla \phi \|_{L^{\infty}(\Si ,g)} \| \partial_t u \|_{L^2(\Sigma,g)},
\eeqa
for a constant $C$ that only depends on an upper bound on the  energy $E_\phi(u,g)$ at the corresponding time. 
\end{lemma}
\begin{proof}[Proof of Lemma \ref{lem:loc-energy}]
Since $\phi$ is time-independent and 
$\partial_t u = \tau_g(u)=P_u(\Delta_g u)$, we have
\beqas
\abs{\ppeps E_\phi(u(\cdot+\eps), g) +\int_{\Sigma} \phi^2 \abs{\partial_t u}^2 dv_g }
&=\abs{ \int_{\Sigma} \partial_t u \langle du, d(\phi^2) \rangle_g dv_g}\leq  CE_\phi(u,g)^\half \| \nabla \phi \|_{L^{\infty}(\Si,g)} \| \partial_t u \|_{L^2(\Sigma,g)},
\eeqas
while the fact that  $\| \phi^2 k(u,g) \|_{L^1(\Si,g)} \leq C E_\phi(u,g)$ allows us to bound 
\beqas
\abs{\ppeps E_\phi(u, g(\cdot+\eps))}&=\mfrac{1}{2}\abs{ 
\int_\Si \phi^2 \langle \partial_t g, k(u,g) \rangle_g dv_g
}\leq  CE_\phi(u,g)\| \partial_t g \|_{L^{\infty}( \text{supp}(\phi),g)}. 
\eeqas
Combined, this yields \eqref{est:loc-energy-1}.
\end{proof}
We also remark that \eqref{est:trivial-norms-gt} and \eqref{est:Met-Var} yield a priori bounds of 
$\| \partial_t g \|_{L^{\infty}(\Si,g)}\leqs \mu^2 $  and hence of
$\abs{\pt \mu}  \leqs \mu^3$ 
 and note that any $C^k$ norm of $\partial_t g$ remains controlled for as long as $\mu$ remains bounded.

Up to minor modifications, we can hence argue exactly as in the original paper \cite{Struwe} of Struwe on the classical harmonic map flow to deduce that for any $(u_0,g_0)\in H^1(\Si,S^2) \times \GG_b$ there exists a solution $(u,g)$ of the 
flow \eqref{eq:flow}, which exists and is smooth (away from $t=0$) on a maximal time interval  $[0,T_{\max})$ where $T_{\max}$ can only be finite if either 
\vspace{-1pc}
\begin{itemize}
    \item $\mu$ leaves the admissible range of parameters, i.e. is so that $\mu(t)\to \infty $ or $\mu(t)\to \mu_*$  as $t \upto T_{\max}$, or
    \item $\mu(t)$ converges to some number $\mu(T_{\max}) \in (\mu_*,\infty)$ as $t \upto T_{\max}$, but the map becomes singular due to the formation of at least one bubble.
\end{itemize}

In the following we will focus on the analysis of solutions of \eqref{eq:flow} with small energy defect for which the initial metric is chosen so that condition \eqref{eqn:exact-bubble}, which we recall in \eqref{def:flow-scale} below, is satisfied, compare also Remark \ref{rmk:SED-concentrate-bubble}.

A key feature of our flow is that this condition remains essentially preserved along the flow.  This not only prevents the formation of singularities for the map component, but also allows us to apply the \Loj estimates obtained in Section \ref{sec:Loj-Est} to control the total $L^2$ distance travelled by both the map and metric component. To be more precise, we show
\begin{prop}\label{prop:flow-main}
For every closed orientable surface $(\Si,g_\Si)$ with positive genus there exist constants $\de_2 >0$ and $C$ such that 
the following holds true. Let 
$v: \Si \to S^2$ be any degree $1$ map with $\de_v \leq \delta_2$ and let $\mu_0 \geq \mu_*$ and $b \in \Si$ be so that 
\beq
\label{def:flow-scale}
E(v,B_{R_\Si(\mu_0^{-1})}(b))=\sup_{p \in \Sigma} E(v,B_{R_\Si(\mu_0^{-1})}(p))=2\pi. 
\eeq 
Then the maximal solution $(u(t),g(t)=g_{\mu(t),b})$,  $t \in [0,T_{\max})$, of \eqref{eq:flow} to initial data 
$(u_0,g_0):=T_{\mu_0,b}^*(v,\bar{g}_{\mu_0,b})$ has the following properties.
\begin{enumerate}
    \item\label{item:flow-1} The assumptions of Theorem \ref{thm:Loj-new} are satisfied 
   for $$(\bar{u},\bar g)(t)=(T_{\mu(t),b})_* (u (t),g(t)) \text{ for every } t\in [0,T_{\max})$$
and hence the \Loj estimates \eqref{est:LE-energy}, \eqref{est:LE-mu} and \eqref{est:dist-bound-1} are applicable for the whole flow. 
   \item \label{item:flow-2} The total distance travelled by the flow
is a priori bounded by 
\beq
\label{est:flow-distance-bound-L2}
\int_0^{T_{\max}} \| \partial_t u \|_{L^2(\Si,g)} + \| \partial_t g \|_{L^2(\Si,g)} dt \leq C\de_v^{\half}.
\eeq
\item \label{item:flow-3}
We have  $\mu(t)\to \infty $ as $t\upto T_{\max}$ irrespective of whether $T_{\max}$ is finite or infinite and we  have a uniform lower bound of $\mu(t)\geq c\de_v^{-\half}$ for all $t\in [0,T_{\max})$.
\end{enumerate}
\end{prop}

\begin{rmk}\label{rmk:apriori-mu-zero}
We note that the last statement in particular ensures that 
\beq
\label{est:apriori-mu-zero}
\mu_0^{-1}\leq C \de_v^{\half} \text{ for } \mu_0 \text{ as in \eqref{def:flow-scale}}.
\eeq
\end{rmk}

\begin{proof}[Proof of Proposition \ref{prop:flow-main}]
We show that the proposition holds for a sufficiently small $\de_2>0$, chosen in particular to be less than the numbers $\de_1$ and $\de_*$ from Theorem \ref{thm:Loj-new} and Remark \ref{rmk:no-SED-harmonic}.

As  $\de(t):=\de_{\bar u(t)}=\de_{(u,g)(t)}$  is non increasing along the flow, the assumption on the smallness of the energy defect
required in Theorem \ref{thm:Loj-new} is thus trivially satisfied. 
We also recall that \eqref{ass:Loj-thm} follows once we check that \eqref{ass:check-in-coord-bar} holds for $\bar u_b(t)=u_b(t)\circ T_{\mu(t)}$. 
As $T_\mu x\equiv\mu_*\mu^{-1} x $ on $\DD_{r_2}$  and as $r_3=\mu_*^{-1}<\half r_2$
this condition \eqref{ass:check-in-coord-bar} is equivalent to 
\beq
\label{est:exact-oroof}
E(u_b(t),\DD_{r_3(1+ \eps_2)})\geq 2\pi-\eps_2 \text{ and } \sup_{\abs{y}\leq r_3} E(u_b(t),\DD_{r_3(1-\half \eps_2)}(y)) \leq 2\pi+\eps_2
\eeq
while
our choice of $\mu_0$ and $b$ ensures that 
\beq
\label{est:exact2}
E(u_b(0),\DD_{r_3})=2\pi\geq \sup_{\abs{y}\leq r_3} E(u_b(0),\DD_{r_3(1-\frac14 \eps_2)}(y)),
\eeq
compare Remark \ref{rmk:ball-disc} and the subsequent discussion.  Hence 
we know that \eqref{est:exact-oroof} holds at least for small times, allowing us to consider 
the maximal $T_1>0$ for which  \eqref{est:exact-oroof} holds on $[0,T_1)$.
This choice of $T_1$ ensures that the \Loj estimate \eqref{est:LE-energy} is applicable on $[0,T_1)$. As already observed by Simon in \cite{Simon} such \Loj estimates can be used to bound the evolution of the energy defect $\de(t):= \de_{(u,g)(t)}$ by
\beqa\label{est:speed-by-de}
- 2\tfrac{d}{dt} \de^{\half} &= \de^{-\half} (-\ddt E(u,g))=
\de^{-\half} \| \partial_t (u,g) \|_{L^2(\Sigma,g)}^2 =\de^{-\half} \norm{\na \EE_b(u,g)}_{L^2(\Si,g)} \cdot  \| \partial_t (u,g) \|_{L^2(\Sigma,g)} \\
& \geq C^{-1} \| \partial_t (u,g) \|_{L^2(\Sigma,g)}
\eeqa
 and hence the distance travelled by the flow by  
\beq \label{claim:dist-flow}
\int_{t_1}^{t_2} \| \partial_t u \|_{L^2(\Si,g)} + \| \partial_t g \|_{L^2(\Si,g)} dt \leq C( \de(t_1)^{\half}-\de(t_2)^{\half}) \text{ for all } 0\leq t_1\leq t_2<T_1.
\eeq
As part \ref{item:metric-for-free}  of Proposition \ref{prop:metrics-main} ensures that 
$\pt g$ vanishes on $B_{R_\Si(r_2)}(b)$, the local energy estimates from Lemma \ref{lem:loc-energy} hence allow us to bound
\beqa
\label{est:loc-energy-2}
\abs{E_\phi((u,g)(t_2)) - E_\phi((u,g)(t_1))} & \leq C 
\| \nabla \ph \|_{L^{\infty}(\Si,g)}\int_{t_1}^{t_2} \norm{\pt u}_{L^2(\Si,g)}dt +\int_{t_1}^{t_2}\norm{\pt u}^2_{L^2(\Si,g)} dt\\ &\leq C \| \nabla \ph_b \|_{L^{\infty}(\DD_{r_2})} \de_{v}^{\half}+\de_v\leq C (1+\| \nabla \ph_b \|_{L^{\infty}(\DD_{r_2})} )\de_v^{\half}
\eeqa
for every $\phi\in C_c^\infty(B_{R_\Si(r_2)}(b), [0,1])$. 

Given any $\abs{y}\leq r_3=e^{-1}r_2$ we can choose such a $\phi$ with 
$\phi_b\equiv 1 $ on $\DD_{r_3(1-\half \eps_2)}(y)$,  
$\supp(\phi_b)\subset \DD_{r_3(1-\frac14 \eps_2)}(y)$ and $\abs{\nabla\phi_b}\leqs \eps_2^{-1}r_3^{-1}\sim 1$ and use 
\eqref{est:exact2}
to deduce that 
$$ E(u_{b}(t),\DD_{r_3(1-\half \eps_2)}(y)) \leq 2\pi +C\de_v^{\half}\leq 2\pi+\thalf \eps_2 \text{ for } t\in [0,T_1)$$
where the last estimate holds after reducing $\de_2$ if necessary.  Similarly, choosing $\phi$ so that 
$\phi_b\equiv 1 $ on $\DD_{r_3}$, $\supp(\phi_b)\subset \DD_{r_3(1+\eps_2)}$ and $\abs{\nabla\phi_b}\leqs \eps_2^{-1}r_3^{-1}\sim 1$ gives 
$$E(u_{b}(t),\DD_{r_3(1+\eps_2)}) \geq 2\pi - C\de_v^{\half} \geq 2\pi -\thalf\eps_2  \text{ for } t\in [0,T_1).$$
If $T_1$ was strictly less than $T_{\max}$ this would ensure that \eqref{est:exact-oroof} would remain valid on some interval $[0,T_1+\eps)$, $\eps>0$, contradicting the definition of $T_1$. We thus conclude that $T_1=T_{\max}$ which not only establishes \ref{item:flow-1}, but also ensures that the above estimates hold on all of $[0,T_{\max})$. In particular, the second part of the proposition follows from \eqref{claim:dist-flow}, which, when combined with 
 $\| \partial_t g \|_{L^2(\Sigma,g)} \sim \mu^{-2} \abs{\partial_t \mu_t}$, also ensures that 
\beq
\label{est:flow-scale-bound}
\abs{\mu^{-1}(t_1) - \mu^{-1}(t_2)} \leq \int_{t_1}^{t_2} \mu^{-2} \abs{\partial_t \mu_t}dt \leq  \int_{t_1}^{t_2} \| \partial_t g \|_{L^2(\Sigma,g)} dt  
\leqs \de(t_1)
^{\half}-\de(t_2)^{\half},
\eeq
for all $0 \leq t_1\leq t_2< T_{\max}$.
Thus $\mu(t)$ converges  to some $\mu_\infty \in [\mu_*,\infty]$ as  $t\upto T_{\max}$ which we claim must be $\mu_\infty=\infty$. 

We first note that \eqref{est:flow-scale-bound} ensure that 
$\mu^{-1}(t)\leq \mu_0^{-1}+C\de(t)^{\half} \leq \frac14 \mu_*^{-1}+ C\de_1^{\frac12}\leq \half \mu_*^{-1}$,  compare Remark \ref{rmk:SED-concentrate-bubble}, thus  excluding the possibility that  $\mu(t)$ approaches the lower bound of the admissible parameter range. 

To exclude the possibility that $\mu_\infty\in (\mu_*, \infty)$ we now note that the formation of a bubble, be it at finite or infinite $T_{\max}$ is excluded by 
\eqref{est:exact-oroof} as the formation of a bubble requires energy at least $4\pi$ and as $E(u,g)\leq 4\pi+\de_1<5\pi$: A bubble forming at a point  
$a\in B_{R_\Si(\frac32r_3)}(b)$
would contradict the second estimate in 
\eqref{est:exact-oroof}, while a bubble forming at some $a\in \Si\setminus B_{R_\Si(\frac32r_3)}(b)$ would contradict the first estimate in \eqref{est:exact-oroof}.
This immediately excludes the possibility that $\mu_\infty\in (\mu_*, \infty)$ if $T_{\max}$ is finite, as in this case we could continue the flow smoothly past $T_{\max}$. 

If $T_{\max}=\infty$ and $\mu_\infty\in (\mu_*, \infty)$ we could instead use 
that $g(t)$ converges smoothly to $g_{\mu_\infty,b}$ to apply standard 
compactness theory for almost harmonic maps to deduce that $u(t)$ subconverges strongly in $H^2$ to a limiting harmonic map $u_\infty: (\Si,g_{\mu_\infty,b})\to S^2$ with $E(u_\infty,g_{\mu_\infty,b})\leq 4\pi+\de_v<4\pi+\de_*$. This also leads to a contradiction since $u_\infty$ cannot be constant, as this would violate the first estimate in \eqref{est:exact-oroof}, and must thus have energy at least $4\pi+\de_{*}$ as the energy is conformally invariant and as 
$(\Si,g_{\mu_\infty,b})$ is isometric to $(\Si,\bar g_{\mu_\infty,b})$, which in turn is conformal to our original surface $(\Si,g_\Si)$. 

Hence $\mu_\infty=\infty$ and the claimed uniform lower bound on $\mu(t_1)$ follows from \eqref{est:flow-scale-bound} in the limit $t_2\to \infty$.
\end{proof}

\section{Proof of Theorem \ref{thm:main}}\label{sec:thm1-proof}
In this section we use the flow \eqref{eq:flow} to prove our main quantitative stability result, Theorem \ref{thm:main}. 

As all the claims of the theorem are trivially true for maps $v$ whose energy defect is larger than a fixed positive constant, it suffices to consider degree $1$ maps $v \in H^1(\Si, S^2)$  with 
\beqs
\de_v := E(v,g_\Si) - 4\pi \leq \de_0 ,
\eeqs
where $\de_0=\de_0(\Si,g_\Si)>0$ will be chosen below and will in particular be so that $\de_0 \leq \min \{\de_1,\de_2,\de_*, \half \}$
for the constants $\de_{1}$, $\de_2$ and $\de_*$ from Theorem \ref{thm:Loj-new}, Proposition \ref{prop:flow-main} and Remark \ref{rmk:no-SED-harmonic}.

Given such a map $v$ we let $\mu_0$ and $b\in \Si$ be so that  \eqref{def:flow-scale} holds, recalling from  
Remark \ref{rmk:apriori-mu-zero} that 
$\mu_0^{-1}\leqs \de_v^{\half}$.
We then consider the 
maximal solution
 $(u(t),g(t)=g_{\mu(t),b}), t \in [0,T_{\max})$, 
of  \eqref{eq:flow}
with $u(0)=T_{\mu_0,b}^*v$ and $g(0)=  g_{\mu_0,b}=T_{\mu_0,b}^*\bar g_{\mu_0,b}$, for which    Proposition \ref{prop:flow-main} ensures that $\mu(t)\to \infty$ as $t\upto T_{\max}$.

Pushing  $(u,g)(t)$ forward with the  \textit{fixed} diffeomorphism $T_{\mu_0,b}$ hence yields a curve of maps and metrics 
 \beq \label{def:deform-curve }(U(t),G(t)):=(T_{\mu_0,b})_* (u(t),g(t)), \qquad t\in [0,T_{\max})
 \eeq with $U(0)=v$ and $G(0)=\bar g_{\mu_0,b}$ 
 which is so that 
 \beq \label{est:length-G-U}\int_0^{T_{\max}} \| \partial_t U \|_{L^2(\Si,G)} + \| \partial_t G \|_{L^2(\Si,G)} dt =
\int_0^{T_{\max}} \| \partial_t u \|_{L^2(\Si,g)} + \| \partial_t g \|_{L^2(\Si,g)} dt \leq C\de_v^{\half},\eeq
compare \eqref{est:flow-distance-bound-L2}. As $\mu(t)\to \infty$, part \ref{item:metric-for-free} of Proposition \ref{prop:metrics-main} furthermore implies that the 
 metrics $G(t)$ converge as $t\upto T_{\max}$ to the metric
$$G_\infty:=(T_{\mu_0,b})_*g_{\infty,b} \text{ on }
\Om_\infty:=T_{\mu_0,b}(B_{R_\Si(r_\infty)}(b))=B_{ R_\Si(\hat r_\infty)} (b), \quad \hat r_\infty=t_{\mu_0}(r_\infty)\in (r_2,r_1]
,$$ 
which is so that $(\Om_\infty,G_\infty)$ is isometric to a punctured sphere.
Indeed,  the key property \eqref{claim:metrics-identical} of the metics $g(t)=g_{\mu(t),b}$ established in Proposition \ref{prop:metrics-main}, 
combined with the fact that $\mu(t)\to \infty$ as $t\upto T_{\max}$, implies that 
 for any compact subset $K\subset \Om_{\infty} $  the metrics $G(t)$ eventually agree with $G_\infty$ on $K$, i.e. that there exists $T_K<T_{\max}$ for which
  \beq 
  \label{eq:metrics-equal-limit} 
  G(t)\equiv G_\infty \text{ and hence } \pt U(t)=\tau_{G_\infty}(U(t))\text{ on } K\times [T_{K},T_{\max}).\eeq
As Proposition \ref{prop:metrics-main} ensures that 
$G(t)\geq e^{-3}G_\infty$ for every $t$ 
we obtain from \eqref{est:length-G-U} that 
$$\int_0^{T_{\max}} \| \partial_t U \|_{L^2(\Om_\infty,G_\infty)}  dt \leq e^3 \int_0^{T_{\max}} \| \partial_t U \|_{L^2(\Si,G)}  dt \leq C\de_v^{\half},
$$
and hence that  $U(t)$ converges strongly
 in $L^2(\Om_\infty,G_\infty)$  to a limit 
 $U(T_{\max})$ as $t\upto T_{\max}$ which is so that
      \beq
    \label{est:dist-to-w}
    \norm{v-U(T_{\max})}_{L^2(\Om_{\infty},G_\infty )}\leq 
\int_0^{T_{\max}} \norm{\pt U}_{L^2(\Om_\infty,G_\infty)} \leq C \delta_{v}^{\half}.
    \eeq 
    We now note that the functions $U_b(t)$ and $u_b(t)$ that represent $U(t)$ and $u(t)$ in the fixed coordinates $F_b$ are related by $u_b(t)=U_b(t)\circ T_{\mu_0}$ and recall that $T_{\mu_0}(x)\equiv \mu_* \mu_0^{-1}x = r_3^{-1} \mu_0^{-1}x$ on $\DD_{r_2}$. 
The estimate \eqref{est:exact-oroof} obtained in the proof of 
        Proposition \ref{prop:flow-main} hence translates into
\beq
\label{est:energy-big-U}
E(U_b(t),\DD_{\mu_0^{-1}(1+\eps_2)})
\geq 2\pi - \eps_2 \text{ and } \max_{\abs{y} \leq \mu_0^{-1} }E(U_b(t),\DD_{\mu_0^{-1}(1-\half \eps_2)}(y)) \leq 2\pi + \eps_2
\eeq
for all $t\in [0,T_{\max})$.     
As \eqref{eq:metrics-equal-limit} allows us to view $U(t)$ as a solution of the classical harmonic map flow (with fixed metric) on $K\times [T_K,T_{\max})$ for which the formation of a bubble  as $t\upto T_{\max}$ is excluded by 
\eqref{est:energy-big-U} and $E(U(t),G(t))\leq 4\pi+\de_v<5\pi$,
we can apply  the classical results
of Struwe \cite{Struwe}
to deduce that $U(t)$ indeed converges smoothly locally to $U(T_{\max})$ on $(\Om_\infty,G_\infty)$ and that  \eqref{est:energy-big-U} is valid also at $t=T_{\max}$. 

We next explain why our construction ensures that the degree of $U(T_{\max})$ is still $1$. 
For this we can use that $(\Om_\infty\setminus\{b\},G_\infty)$ is conformal to the infinite cylinder $C_\infty=\R\times S^1$ via a diffeomorphism $F_{\infty}^{\text{cyl}}$ which maps (punctured) balls $B_R(b)\setminus \{b\}$ to cylinders $(-\infty, s_{R})\times S^1$.
We now note that given any map $w:C_\infty\to S^2$ with finite energy and any $s^*\in \R$ with $\int_{S^1}\abs{\partial_\theta w(s^*,\theta)}^2 d\theta< \eps\ll 1$, we can cut $w(s^*,\theta)$ off to a constant on unit cylinders and project onto $S^2$ to obtain maps $w^\pm:C_\infty\to S^2$ with $w^\pm\equiv w$ for $\pm (s-s^*)\geq 0$ and $w^\pm$ constant for  $\pm (s-s^*)\leq -1$ so that 
$$\deg(w^+)+\deg(w^-)=\deg(w) \text{ and } E(w^+)+E(w^-)=E(w)+o(1),$$
where $o(1)$ can be made arbitrarily small by reducing $\eps$, compare \cite[Section $3$]{R-Rig-23}. 
The smooth local convergence of the maps $w(t):=U(t)\circ (F_{\infty}^{\text{cyl}})^{-1}$ to $w(T_{\max}):=U(T_{\max})\circ (F_{\infty}^{\text{cyl}})^{-1}$ means that we 
can carry out this construction with a suitably chosen fixed $s^*> s_{R_2}$ for all $t\leq T_{\max}$ that are sufficiently close to $T_{\max}$. The resulting maps $U^\pm(t)=w^{\pm}(t)\circ (F_{\infty}^{\text{cyl}})$ are so that $E(U^+(t))\leq E(u(t), \Si\setminus B_{R_2}(b))<3\pi$, ensuring that  $\deg(U^+(t))=0$, and so that $U^-(t)\to U^-(T_{\max})$ smoothly on $\Si$. We thus conclude that 
$\deg(U(T_{\max}))=\deg(U^-(T_{\max}))=\lim_{t\upto T_{\max}}\deg(U^-(t))=\lim_{t\upto T_{\max}}\deg(U(t))=1$.

We note that \eqref{eq:metrics-equal-limit}
ensures that $U(T_{\max})$ is harmonic if $T_{\max}=\infty$. If $T_{\max}<\infty$  we instead use  $U(T_{\max})$ as initial data for the standard harmonic map flow on the sphere, i.e. extend our flow past $T_{\max}$ as 
$$\pt U(t)=\tau_{G_\infty}(U(t)), \qquad G(t)\equiv G_\infty \text{ on } \Om_\infty\times [T_{\max},\infty).
$$ 
Instead of using the \Loj estimate \eqref{est:LE-energy} from Theorem \ref{thm:Loj-new}, we can now use the 
 \Loj estimate for maps from $S^2$ to itself obtained by Topping in \cite{Topping-rigidity} which allows us to bound 
\beq E(U(t),G_\infty)-4\pi\leq C \norm{\tau_{G_\infty}(U(t))}_{L^2(\Om_\infty,G_\infty)}^2 \label{est:Loj-Topping} \eeq  for as long as  the solution $U(t)$, $t>T_{\max}$, remains smooth.   
After reducing $\de_0$ if necessary, we can hence repeat the argument used in the proof of Proposition \ref{prop:flow-main} to deduce that the solution remains smooth and 
the estimates 
\beq
\label{est:E-small-after-sing}
E(U_b(t),\DD_{\mu_0^{-1}(1+2\eps_2)})
\geq 2\pi - 2\eps_2 \text{ and } \max_{\abs{y} \leq \mu_0^{-1} }E(U_b(t),\DD_{\mu_0^{-1}(1-2 \eps_2)}(y)) \leq 2\pi + 2\eps_2
\eeq
remain valid  on all of $[T_{\max},\infty)$. Hence, as $t\to \infty$, the maps $U(t)$ converge smoothly to a degree $1$ harmonic map $U(\infty)$ which satisfies  
\eqref{est:E-small-after-sing} and, thanks to 
\eqref{est:Loj-Topping} and 
\eqref{est:speed-by-de}, 
$$ \norm{U(T_{\max})-U(\infty)}_{L^2(\Om_\infty, G_\infty)}\leq \int_{T_{\max}}^\infty \norm{\pt U}_{L^2(\Om_\infty, G_\infty)} dt\leq C (E(U(T_{\max}),G_\infty)-4\pi)^{\half}\leq C\de_v^{\half}.$$
Irrespective of whether  $T_{\max}$ is finite or infinite, we can thus conclude that $U_\infty:=\lim_{t\to \infty} U(t)$ is a harmonic map from the punctured sphere $(\Om_\infty,G_\infty)$ to $S^2$ which satisfies \eqref{est:E-small-after-sing} and 
\beq
\label{est:dist-L2-infinite}
\norm{v-U_\infty}_{L^2(\Om_\infty, G_\infty)}\leq C\de_v^{\half}.
\eeq
We now note that 
\eqref{claim:metrics-identical} (applied for $\mu_1=\mu_0$) ensures  that
 $$G_\infty\equiv G(0)\equiv \bar g_{\mu_0,b}=
F_b^*g_{\pi,\mu_0} \text{ on } B_{R_2}(b),$$
meaning that we can  map
$ (\Om_\infty, G_\infty)$ isometrically to $(\R^2,g_{\pi,\mu_0})$ using a diffeomorphism that is given by the identity in the usual conformal coordinates on $B_{R_2}(b)\subset \Om_\infty$. The standard description of degree $1$ harmonic maps from $\R^2$ to $S^2$ hence allows us to write $U_\infty$ in the usual coordinates on $B_{R_2}(b)$ as 
\beq 
U_{\infty,b}:= U_\infty\circ F_b^{-1}=\RR \pi(\la(x-x_0)) \text{ for some } x_0\in \R^2,  \la>0 \text{ and } \RR\in SO(3) \label{eq:writing-Uinfty}\eeq
and thus deduce that 
\beq 
\label{est:sunny-1}
\| v_b -   \RR \pi(\la(x-x_0)) \|_{L^2(\DD_{r_2}, g_{\pi,\mu_0})} \leq C \de_v^{\half}.\eeq
Since the energy estimate \eqref{est:E-small-after-sing} also applies for $U_{\infty,b}$
we immediately see that 
\beq 
\label{est:la-a-ok}
\abs{x_0}\leqs \mu_0^{-1} \text{ and } 
\la\sim \mu_0 \text{ and hence in particular } \abs{x_0}+\la^{-1}\leqs \de_v^\half,
\eeq
compare Remark \ref{rmk:apriori-mu-zero}.

We now prove that the claims of the theorem holds for this choice of $\la$, $\RR$ and  $a := F_b^{-1}(x_0)$, noting that it suffices to carry out this argument
in the case where $\RR=\text{Id}$ (as we can otherwise replace $v$ by $\RR^{-1} v$) and that 
 \eqref{claim:la} already follows from \eqref{est:la-a-ok}. 
 
 We begin by proving that $v_a:=v\circ F_a^{-1}$
satisfies 
\beqa
\label{est:weighted-L2}
\| v_a - \pi_\la \|_{L^2(\DD_{\iota},g_{\pi,\la} )} \leq C\de_v^{\half}.
\eeqa
To see this we note that since 
\eqref{est:la-a-ok} implies that $d_{g_\Si}(a,b)\leqs \la^{-1}$ is small, 
the transition map between the conformal coordinate charts at $a$ and $b$ is essentially given by a translation, namely so that 
$F_{b,a}:=F_a\circ F_b^{-1}=Id-x_0+\text{err}_{b,a}$ for an error term that
satisfies
$\abs{\err_{b,a}(x)}\leqs d_{g_\Si}^2(a,b)\abs{x} +d_{g_\Si}(a,b)\abs{x}^2\leq \la^{-1} \abs{x}$, see \eqref{est:err-bound}.
As $\abs{\na \pi_\la(x)}\sim \abs{\na \pi_\la(\tilde x)}$ for all $x,\tilde x$ with $\abs{x-\tilde x}\leqs \la^{-1}$, so in particular for points $\tilde x$ on the segment connecting  $x-x_0$ and $F_{b,a}(x)$, we can hence bound 
$$ \abs{\pi_\la(F_{b,a}(x))- \pi_\la(x-x_0) }\leqs 
\abs{ \err_{b,a}(x)}  \abs{\na\pi_\la(x)} \leqs \la^{-1}\abs{x
 }  \abs{\na\pi_\la(x)}\leqs \la^{-1}.
$$
As \eqref{est:la-a-ok} ensures that 
$\bar g_{\mu_0,b}\sim \bar g_{\la,a}$ and that $B_{R_{\Si}(\half r_2)}(a)\subset B_{R_{\Si}(r_2)}(b)$ we can combine this bound with \eqref{est:sunny-1} and \eqref{est:la-a-ok} to deduce that 
\beqas
\| v_a -   \pi_\la \|_{L^2(\DD_{\half r_2},g_{\pi,\la})} 
&=\| v -   \pi_\la\circ F_a \|_{L^2(B_{R_\Si(\half r_2)}(a),\bar{g}_{\la,a})}
\\ &\leqs \| v -   \pi_\la\circ F_a \|_{L^2(B_{R_\Si(r_2)}(b),\bar{g}_{\mu_0,b})}  = \| v_b -   \pi_\la\circ F_{b,a} \|_{L^2(\DD_{r_2},g_{\pi,\mu_0})} 
\\
& \leqs \| v_b -   \pi(\la(x-x_0)) \|_{L^2(\DD_{r_2},g_{\pi,\mu_0})} + \la^{-1} \leq C\de_v^{\half}.
\eeqas
Since $\Area_{g_{\pi,\la}}(\DD_\iota\setminus \DD_{\half r_2})\leqs \la^{-2}\leqs \de_v$, compare \eqref{est:pmu-bulk} and \eqref{est:la-a-ok}, and since we can trivially bound $\norm{v_a-  \pi_\la}_{L^\infty}\leq 2$ this yields the 
estimate \eqref{est:weighted-L2} claimed above, i.e. the $L^2$ version of the $H^1$-estimate \eqref{claim:thm-qs1} claimed in the theorem. 

To carry out the proof of the remaining claims we now show that 
the estimate 
\beq \label{claim:I-main-proof}
\int_{\Si} \abs{\na v-\na \bar z}^2 dv_{g_\Si}\leq C\de_v
\eeq
holds true for the 
 map $ \bar z:=\bar z_{\la,a}:\Si\to S^2$ that was constructed by the first author in \cite{R-Loj-1}, and whose precise definition is recalled in Section \ref{subsec:sing-models}.
 This degree $1$ map $\bar z:\Si\to S^2$ is essentially given by  $\pi_\la $ in the usual coordinates on $B_{I}(a)$ in the sense that 
\beq \label{est:sing-model-conf} 
\| \bar{z}_a -  \pi_\la \|_{C^1(\DD_{\iota})} \leqs \la^{-1} \quad \text{and} \quad \abs{\Delta \bar{z}_a - \Delta  \pi_\la} \leqs  \la^{-1}\abs{\nabla \pi_\la}+\la^{-2} ,\eeq
and so that $\norm{\bar z-\pi(\infty)}_{C^2(\Si \setminus B_{I}(a))}\leqs \la^{-1}$ and 
$\norm{\Delta_{g_\Si} \bar{z}}_{C^0(\Si \setminus B_{I}(a))} \leqs \la^{-2}$, compare \cite[Section 3]{R-Loj-1}, in particular estimate (3.15). 
Here we continue to use 
the convention that quantities on $\Si$ are computed using  $g_\Si$, while quantities on discs are computed using $dx^2$ unless indicated otherwise. 

Similar to \cite{Topping-deg-1} and \cite{R-Rig-23} we now estimate 
\beqa\label{est:I-main-proof}
I 
&:=  \int_{\Si} \abs{\na v-\na \bar z}^2= \int_{\Si} \nabla (v-\bar{z})\na(v+\bar{z}-2\bar z) = 2\int_{\Si} \Delta\bar{z} \cdot (v - \bar{z})  + 2( E(v) - E(\bar{z})) \\
& \leq 2\int_{\DD_{\iota}} \Delta \bar z_a \cdot (v_a - \bar{z}_a) +4 
\norm{\Delta \bar{z} }_{L^\infty(\Si \setminus B_{I}(a))}+ 2 \de_v   \\
&\leq 2\int_{\DD_{\iota}}  \Delta \pi_\la \cdot (v_a - \bar{z}_a) +C\la^{-1} \int_{\DD_{\iota}} \abs{\nabla \pi_\la} \abs{v_a - \bar{z}_a} + C \la^{-2} +2 \de_v \\
&\leq -2\int_{\DD_{\iota}} \abs{\na\pi_\la}^2\pi_\la \cdot (v_a - \bar{z}_a)
+C \| v_a - \bar{z}_a\|_{L^2(\DD_{\iota},g_{\pi,\la} )}^2 +C\la^{-2} +2\de_v.
\eeqa 
Since  $v_a$, $\bar z_a$ and $\pi_\la$ all map into $S^2$ we have 
pointwise estimates of 
$\abs{P_{\pi_\la}^{\perp}(v_a-\pi_\la)} \leqs  \abs{v_a- \pi_\la}^2$ and $\abs{P_{\pi_\la}^{\perp}(\bar{z}_a-\pi_\la)} \leqs\abs{\bar{z}_a- \pi_\la}^2\leqs \la^{-2}$ for the projections of these vectors onto the corresponding normal space $T^{\perp}_{\pi_{\la}}S^2=\text{span}(\pi_\la)$.  
Combined with \eqref{est:sing-model-conf}, \eqref{est:I-main-proof} 
and the already established bounds \eqref{est:weighted-L2}
and $\la^{-1}\leqs \de_v$ this yields that 
$I\leqs \| v_a - \pi_\la\|_{L^2(\DD_{\iota},g_{\pi,\la} )}^2 +\la^{-2} +\de_v\leqs \de_v$ as claimed.

From  \eqref{claim:I-main-proof}, \eqref{est:sing-model-conf} and the conformal invariance of the energy we immediately deduce that
$$ \norm{v_a -  \pi_{\la}}_{\dot H^1(\DD_{\iota},g_{\pi,\la}) }= \norm{\na (v_a -  \pi_{\la})}_{L^2(\DD_\iota)}\leqs 
\norm{\na (v_a -  \bar z_a)}_{L^2(\DD_\iota)}+\norm{\na (\bar z_a -  \pi_{\la})}_{L^2(\DD_\iota)}
\leqs \de_v^\half 
,$$ thus completing the proof of \eqref{claim:thm-qs1}, and obtain \eqref{claim:constant-H01} as 
\beqa 
\|  \na v\|_{L^2(\Sigma \setminus B_{R}(a) ,g_\Si)} &\leq 
\norm{\na (v -  \bar z)}_{L^2(\Sigma \setminus B_{R}(a) ,g_\Si)}+ \| \na \bar z\|_{L^2(\Sigma \setminus B_{R}(a) ,g_\Si)} \\ &\leq C\de_v^\half+C\la^{-1} +\norm{\na \pi_\la}_{L^2(\DD_\iota\setminus \DD_{r_\Si(R)})}
\\
&\leq C \de_v^\half+C(1+r_\Si(R)^{-1})\la^{-1}
\leqs R^{-1} \de_v^\half, \qquad \text{for all $R \in (0,I)$. }
\eeqa
To establish \eqref{claim:constant-L2} we 
first note that 
$\big|\int_{\DD_\iota} [\pi_\la(x)-\pi(\infty) ]\rho_\Si^2 dx\big|\leqs \int_{\DD_\iota}(1+\la^2\abs{x}^2)^{-1} dx\leqs \la^{-2}\log\la\leq \la^{-1},$
as  $\rho_\Si(x)=\rho_\Si(\abs{x})\sim 1$, 
$\pi^{1,2}(\infty)=0=\int_{\partial \DD_r}\pi_\la^{1,2}$ for every $r\in (0,\iota)$ and $\abs{\pi_\la^3(x)-\pi(\infty)}\leqs (1+\la^2\abs{x}^2)^{-1}$. 
 As  $\bar z_a=\pi_\la+O(\la^{-1})$ on $\DD_\iota$ and $\bar z=\pi(\infty)+O(\la^{-1})$ on $\Si\setminus B_{I}(a)$  this ensures that 
$m_{\bar z}:= \fint_\Si \bar z dv_{g_\Si}=\pi(\infty)+O(\la^{-1})$. 
 Setting $m_v:= \fint_\Si v dv_{g_\Si}$, we can thus combine \eqref{claim:I-main-proof} with the Poincar\'e inequality on the fixed surface $(\Si,g_\Si)$ 
to bound 
\beqas
\norm{v-m_v}_{L^2(\Si,g_\Si)}
&\leq 
\norm{v-\bar z-(m_v-m_{\bar z})}_{L^2(\Si,g_\Si)}
+
\norm{\bar z-m_{\bar z}}_{L^2(\Si,g_\Si)}\\
&\leqs
\norm{\na(v-\bar z)}_{L^2(\Si,g_\Si)} 
+
\norm{\bar z-\pi(\infty)}_{L^2(\Si,g_\Si)}+\la^{-1}\\
&\leqs 
\de_v^\half
+\norm{\pi_\la-\pi(\infty)}_{L^2(\DD_\iota, \rho_{\Si}^2 dx^2)}+\la^{-1}  \leqs  \de_v^\half+\la^{-1} (1+(\log\la)^\half) \leqs \de_v^\half \abs{\log \de_v}^\half.
\eeqas
As $\abs{v}=1$ on $\Si$ this also ensures that 
$\abs{1-\abs{m_v}}
\leq (\text{Area}_{g_\Si}(\Si))^{-1} 
\norm{v-m_v}_{L^1(\Si,g_\Si)} 
\leqs \de_v^\half \abs{\log \de_v}^\half 
$
and hence that \eqref{claim:constant-L2} holds for $p:=\frac{m_v}{\abs{m_v}}$. 

This completes the proof of our main theorem, up to the proof of the \Loj estimates from Theorem \ref{thm:Loj-NHD} which is carried out in the remaining part of this paper.

\section{Proof of Theorem \ref{thm:Loj-NHD}}\label{sec:proof-LojEst}
In this section we prove the \Loj estimates stated in Theorem \ref{thm:Loj-NHD}, based on some technical estimates whose proof will be given in the subsequent 
Section \ref{sec:technical-proofs}. 
In the current and the next section we will use the 
 convention that all statements are to be understood to hold true provided the numbers $\la_0>1$ and $\eps_1>0$ that appear in the theorem are chosen sufficiently large, respectively small, and will furthermore use the shorthand $o(1)$ to denote quantities that can be made smaller than any given positive constant by choosing $\eps_1$ (and later also $\eps_0$) small enough and $\la_0$ large enough.

The key idea of the proof is to use dimension reduction techniques and carry out a precise analysis of the energy and its variation on a suitably chosen finite dimensional set of models following the approach of \cite{MRS}, where in the present setting the models are given by suitably chosen pairs of maps and metrics. 
For this we want to work with a 
$6$ parameter family 
\beq
\label{def:Z}
\ZZ:=\{\RR \bar z_{\la,a}, \quad \RR\in SO(3), \,a\in \Si,\, \la\geq 2 \la_0 \}
\eeq
of maps which are so that 
\beq
\label{eq:z-approx}
\bar z_{\la,a}\approx \pi_\la \circ F_a \text{ on } B_{I}(a) \text{ and } \bar z_{\la,a}\approx \pi(\infty)\text{ on } \Si\setminus B_{I}(a).
\eeq
We note that maps $\bar z_{\la,a}:\Si\to S^2$ for which \eqref{eq:z-approx} holds up to carefully chosen modifications of order $\la^{-1}$ in $C^1(\Si,g_\Si)$ were constructed by the first author in \cite{R-Loj-1} in a way that yields the crucial bounds 
\beq  \label{est:energy-old}
\abs{E(\bar{z}_{\lambda,a})-4\pi} \leqs\la^{-2}, 
\qquad 
 -\tfrac{d}{d\lambda} E(\bar{z}_{\lambda,a}) \sim   \lambda^{-3}
 \eeq
and
\beq \label{est:energy-old-var}
 \abs{dE(\bar z_{\la,a})(w)}\leqs  \la^{-2}(\log\la)^\half \norm{w}_{H^1(\Si,\bar g_{\la,a})} \text{ for all }w\in \Gamma^{H^1}(\bar z_{\la,a}^* TS^2).
\eeq
Another key feature of $\ZZ$  concerns the behaviour of the second variation of the energy 
\beq
\label{def:2nd-var}
d^2 E(u)(v,w) := \ddeps \tfrac{d}{d \delta} E(\pi_{S^2}(u + \eps v+\de w))
=\int_{\Si} \nabla v \cdot \nabla w - \abs{\nabla u}^2 v\cdot w \,dv_{g_\Si}, v,w \in \Ga^{H^1 \cap L^{\infty}}( u^* TS^2).
\eeq
To make this precise, we note that we are working in a setting where the bubbles $\om:S^2\to S^2$ which we glue onto constant maps are minimisers, rather than more general critical points as considered in \cite{R-Loj-1,R-Loj-2}, and that the kernel of $d^2E(\om)$ coincides with the tangent space to the set of these bubbles, i.e. degree $1$ M\"{o}bius maps. Thus the construction of the maps $\bar z_{\la,a}$ in \cite{R-Loj-1} ensures that 
the restriction of $d^2E(\bar z_{\la,a})$ to the 
 $H^1(\Si,\bar g_{\la,a})$-orthogonal complement $T_{\bar z}^\perp \ZZ$ 
 of 
 $T_{\bar z} \ZZ$ in $\Gamma^{H^1}(\bar z^* TS^2)$
 is uniformly positive definite. 

\begin{rmk}\label{rmk:E-var}
While we will only ever evaluate $d^2 E(u)(v,w)$ for $v,w$ which are tangent to $S^2$ at $u$, it will at times be useful to extend the definition of $dE(u)$ to general vector fields, and we will do this by extending $dE$ by zero in normal directions, i.e. setting $dE(u)(f) := dE(u)(P_u f)$. 
We note that this allows us to write  
 $\ddeps E(\pi_{S^2}(u+\eps f)) = dE(u)(f) $ for general $f \in H^1(\Si,\R^3)$ 
and hence to compute
\beq
\label{rmk:energy-var}
\ddt (dE(u_t)(f_t)) = d^2E(u_t)(\partial_t u_t, f_t) + dE(u_t)(\pt f_t) ,
\eeq
for any 
family of maps $u_t: \Si \to S^2$ and corresponding tangent fields $f_t \in \Ga^{H^1}(u_t^* TS^2)$. We note that for vector fields $f_t:= P_{u_t} f$ which are obtained 
by projecting a fixed $f\in H^1(\Si,\R^3)$
the second term 
$dE(u_t)(\pt f_t)= dE(u_t)(P_{u_t}(\pt (P_{u_t} f)))$ cannot be expected to vanish unless $f$ happens to tangential to $T_{u_t} S^2$ at the corresponding time $t$.
\end{rmk}

In addition to these maps $\bar z_{\la,a}$ we will use 
the weighted metrics $\gbarmb$ 
defined in Section \ref{sec:metrics}. 
While we will not impose that the parameters $(\la,a)$ of the map and $(\mu,b)$ of the metric coincide, we restrict our attention to pairs of parameter which are compatible in the sense that 
\beq
\label{est:compatible-for-free}
 \la^{-1}\mu \in [1-\eps_0,1+\eps_0] \text{ and }  d_{a,b}:= d_{\Si}(a,b) \leq \eps_0 \la^{-1},
\eeq
 for a small number $\eps_0>0$ for which we use the convention described at the start of the section. We note that 
  this compatibility condition and the construction of the maps $\bar z_{\la,a}$ in particular ensure that 
\beq
\label{est:na-z-1}
\abs{\na_{g_\Si} \bar z_{\la,a}}\leqs \bar{\rh}_{\la,a} \text{ and }
\bar{\rh}_{\la,a}\sim \bar{\rh}_{\mu,b}.
\eeq
As in the rest of the paper it will be crucial to view the energy as a function of both map and metric and to exploit the flexibility that is provided by working either with pairs $(\bar z_{\la,a},\bar g_{\mu,b} ) $ for which the metric is conformal to $g_{\Si}$ or their pull-backs
 $(
 \zmbla ,g_{\mu,b})=T_{\mu,b}^*(\bar z_{\la,a},\bar g_{\mu,b} )$, $T_{\mu,b}$ as in Section \ref{sec:metrics}, as convenient. 

We note that \eqref{est:energy-old-var} already provides a bound on $dE(\zmbla,g_{\mu,b})(f,0)=dE(\bar z_{\la,a},\bar g_{\mu,b})(\bar f,0)=dE(\bar z_{\la,a})(\bar f)$ and will furthermore show the following estimate for the second variation

\begin{lemma}\label{lem:sec-var-map}
 Let  $a,b,\mu,\la$ be parameters which are compatible as described in \eqref{est:compatible-for-free}. Then we can bound 
 \beq
\label{est:sec-var-map}
\abs{d^2E(z, g)(\peps (z_\eps, g_\eps),(f,0))} \leq C \big[ 
\la^{-3} (\log \la)^{\half}\abs{\peps \la}+\la^{-1} \abs{\peps a}+\la^{-3} \abs{\peps \mu_\eps}\big],
\eeq 
for any variation 
 $(z_\eps,g_\eps)=T_{\mu_\eps,b}^*(\bar z_{\la_\eps,a_\eps} ,\bar g_{\mu_\eps,b})$
 and any $f\in\Gamma^{H^1}(z^* TS^2)$ with $\| f \|_{H^1(\Si,g)}=1$. 

\end{lemma}

We note that $d^2E(z,g)(\peps (z_\eps, g_\eps),(f,0)):=
\ddeps \ddelta E(\pi_{S^2}(z_\eps +\de f),g_\eps) =\ddeps [  dE(z_\eps,g_\eps)(f,0)]
$
 as $f$ is assumed to be tangential to $S^2$ at $z=z_{\eps=0}$, compare Remark \ref{rmk:E-var}.

As the metrics $\bar g_{\mu,b}$ are all conformal to $g_\Si$ and as rotations are isometries we know that 
$ 
E(\RR 
\zmbla,g_{\mu,b})=E(\RR \bar z_{\la,a},\bar g_{\mu,b})=E(\bar z_{\la,a})
$ 
only depends 
on $\la$ and $a$. The dependence on $a$ indeed turns out to be of lower order,  namely so that 
\beq 
\label{est:claim-peps-za}
\abs{\tfrac{d}{d\eps} E(\bar z_{\la,a_\eps})} \leqs \lambda^{-2} \abs{\peps a},
\eeq compare Section \ref{subsec:dEda}. 

Variations for which $\peps \la \sim -1$ and for which $\abs{\peps a_\eps}\leqs o(1) \la^{-1}$ hence result in a change of the energy 
 of 
 \beq \label{eq:change-E-main-dir}
 \ddeps E(z_\eps,g) \sim \la^{-3}.
 \eeq
 While such variations of $\la$ always result in a change of 
$\bar z_{\la,a}$ of order at least $\la^{-1}$, we can ensure that the change of the pulled-back maps and metrics is far smaller as described in the following lemma.

\begin{lemma}
    \label{lemma:special-variation}
    Let  $a,b,\mu,\la$ be parameters which are compatible as described in \eqref{est:compatible-for-free} and let 
    $a_\eps, \mu_\eps, \la_\eps$ be variations that satisfy
    \beq
\label{eqn:v-par-def}
\peps \mu_\eps = -1, \quad
\peps (\mu_\eps^{-1} \lambda_\eps) =0\text{ and } \ppeps(\lambda_\eps F_b(a_\eps)) = 0.
\eeq
Then 
 $(z_\eps,g_\eps)=T_{\mu_\eps,b}^*(\bar z_{\la_\eps,a_\eps} ,\bar g_{\mu_\eps,b})$ satisfies
\beq 
\label{est:special-var}
\| \ppeps z_\eps \|_{H^1(\Si, g)} +\| \ppeps z_\eps \|_{L^{\infty}(\Si)}+ \| \ppeps g_\eps \|_{L^2(\Si, g)} \leqs \la^{-2}.
 \eeq 
\end{lemma}
We note that the estimate for  $ \| \ppeps g_\eps \|_{L^2(\Si, g)} $ is an immediate consequence of Proposition \ref{prop:metrics-main} and postpone the proof of the bounds on $\peps z_\eps$ to Section \ref{subsec:proof-lemma-norm-var}.

We note that variations $a_\eps$ of $a$ for which \eqref{eqn:v-par-def} is satisfied are guaranteed to exist and to be so that 
\beq\label{est:spe-var-a} \abs{\peps a}\leq o(1)\la^{-2}\eeq
as 
 $dF_b(b)$ is the identity and as $d_{a,b}\leq \eps_0 \la^{-1} =o(1)\la^{-1}$, compare \eqref{est:compatible-for-free}.

We recall that $T_{\mu,b}$ is given by a dilation $x\mapsto \mu^{-1}\mu_* x$ in the usual conformal coordinates on  $B_{R_2}(b)$ and note that the transition map 
$F_{b,a}=F_a\circ F_b^{-1}$ between the local coordinate charts at $b$ and $a$ is well approximated by the translation $x-F_b(a)$, compare \eqref{def:err-ab}. Thus 
$\zmbla \circ F_b^{-1}$  is
well approximated by the function
$  \pi(\la(\mu^{-1}\mu_* x - F_b(a)))$  whose variation vanishes if we vary the parameters as specified in  
\eqref{eqn:v-par-def}.

We postpone the proofs of \eqref{est:claim-peps-za} and of these two lemmas to Sections \ref{subsec:dEda}-\ref{subsec:sec-var} and first  complete the proof of our key \Loj estimate. 

\begin{proof}[Proof of Theorem \ref{thm:Loj-NHD}]
Let $\bar u:\Si\to \Si$ be as in the theorem, i.e. so that there exists
a point $b\in \Si$, a rotation $\RR_0\in SO(3)$ and a scale $\mu\geq \la_0$ for which 
\eqref{ass:LE-map-1} and \eqref{ass:LE-map-2} hold. 

We let $\bar z =\RR \bar z_{\la,a}\in\ZZ$ be an element of $\ZZ$ which minimises $(\tilde \RR,\tilde \la,\tilde a)\mapsto \norm{\bar u-\tilde \RR \bar z_{\tilde \la,\tilde a}}_{H^1(\Si,\bar g_{\tilde \la,\tilde a})}$, noting that the existence of such a map $\bar z$ is ensured by Lemma 4.2 of  \cite{R-Loj-1} which furthermore guarantees that  $\norm{\bar u-\bar z}_{L^\infty(\Si)}=o(1)$.

As $\norm{\bar z-\RR_0 \bar z_{\mu,b}}_{L^\infty(\Si)}\leq \norm{\bar u-\RR_0 \bar z_{\mu,b}}_{L^\infty(\Si)}+\norm{\bar u-\bar z}_{L^\infty(\Si)}=o(1)$, compare \eqref{ass:LE-map-1}, \eqref{ass:LE-map-2} and  \eqref{eq:z-approx}, 
the parameters $(\la,a)$ of this nearest model $\bar z=\RR\bar z_{\la,a}$ are  compatible with the given parameters $(\mu, b)$  in the sense of \eqref{est:compatible-for-free}
and also $\norm{\RR-\RR_0}=o(1)$.

The two main steps in the proof of the theorem are now to establish that
\beqa
\label{claim:step1}
\norm{\bar u-\bar z}_{H^1(\Si,\barg_{\la,a})}
\leqs 
\norm{\tau_{\bar{g}_{\mu,b}}(\bar{u}) }_{L^2(\Si, \bar{g}_{\mu,b})} +\la^{-2}(\log\la)^\half
\eeqa
for this $\bar z$ 
as well as that
\beq
\label{claim:step2}\la^{-1}\leqs \norm{\na \EE_b(u,g_{\mu,b})}_{L^2(\Si,g_{\mu,b})}+\norm{\bar u-\bar z}_{H^1(\Si,\barg_{\la,a})} \text{ for } u:= \bar u\circ T_{\mu,b},
\eeq
recalling from \eqref{def:EG} that $\norm{\na \EE_b(u,g_{\mu,b})}_{L^2(\Si,g_{\mu,b})}^2= \| \tau_{\bar{g}_{\mu,b}}(\bar{u}) \|_{L^2(\Sigma,\bar{g}_{\mu,b})}^2+\frac14\| \bar P_{\bar{g}_{\mu,b}}(k(\bar{u},\bar g_{\mu,b}))\|_{L^2(\Sigma,\bar g_{\mu,b})}^2$. 

We will provide the proofs of these key estimates below and first explain how they imply the theorem.

To this end, we note that the metrics 
$\gbarla$ and $\gbarmb$ are conformal to each other and, thanks to  \eqref{est:compatible-for-free}, so that $\gbarla \sim \gbarmb$. Hence $\norm{\tau_{\bar{g}_{\la,a}}(\bar{u}) }_{L^2(\Si, \bar{g}_{\la,a})}\sim 
\norm{\tau_{\bar{g}_{\mu,b}}(\bar{u}) }_{L^2(\Si, \bar{g}_{\mu,b})}\leq \norm{\na \EE_b(u,g_{\mu,b})}_{L^2(\Si,g_{\mu,b})}$. 

The above two key estimates \eqref{claim:step1} and \eqref{claim:step2}
hence imply that 
\beqa
\label{claim:step1-2-conseq}
\norm{\bar u-\bar z}_{H^1(\Si,\barg_{\la,a})}+\la^{-1}
\leqs \norm{\na \EE_b(u,g_{\mu,b})}_{L^2(\Si,g_{\mu,b})}.
\eeqa
As \eqref{est:compatible-for-free} in particular ensures that $\la\sim \mu$, we thus immediately obtain the claimed  
bound \eqref{est:LE-mu} on $\mu^{-1}$.

Similarly, the  
claimed distance \Loj estimate \eqref{est:dist-bound-1} immediately follow from the above bound on $\bar u-\bar z$ since \eqref{eq:z-approx} holds up to errors  of order $O(\la^{-1})$ in $C^1(\Si,g_\Si)$.

In order to show the final claim \eqref{est:LE-energy} we set $\bar w:= \bar u -\bar z$ and interpolate between $\bar u$ and $\bar z$ using the family of maps  
$ \bar u_t=\pi_{S^2}(\bar u+t\bar w) $. 
We note that $\bar w_t:= \partial_t \bar u_t=\abs{\bar u+t\bar w}^{-1} P_{\bar u_t}(\bar w)$, in particular $\bar w_0=P_{\bar u} \bar w\in \Gamma(\bar u^*TS^2)$, and that a short calculation, using \eqref{est:na-z-1} and $\abs{\bar u}=1$, allows us to bound 
\beq \label{est:w-hallo-1} \abs{\na (\bar u_t-\bar u_{s})}\leqs  \bar \rho_{\mu,b}\abs{\bar w}+\abs{\na \bar w},\quad \abs{\bar w_t-\bar w_s}\leqs \abs{\bar w}^2 \text{ and } \abs{\na(\bar w_t-\bar w_s)}\leqs \bar \rho_{\mu,b} \abs{\bar w}^2+\abs{\na \bar w} \abs{\bar w}
\eeq
for any $s,t\in [0,1]$. 
We now write
\beq
\label{est:step3}
E(\bar u)-4\pi=E(\bar z)-4\pi+dE(\bar u)(P_{\bar u}(\bar w))+\err_1
\eeq 
for
$\err_1:= \int_0^1dE(\bar u_t)(\bar w_t) -dE(\bar u_0)(\bar w_0) 
dt =\int_0^1\int_{\Si}\na \bar u_t\na \bar w_t-\na \bar u_0 \na \bar w_0 dv_{g_\Si} dt$ which, by \eqref{est:w-hallo-1}, is bounded by $\abs{\err_1}\leqs  \int_{\Si}\abs{\nabla \bar w}^2+\bar \rho_{\mu,b}^2 \abs{\bar w}^2 dv_{g_\Si}=
\norm{\bar w}_{H^1(\Si,\bar g_{\mu,b})}^2\sim \norm{\bar w}_{H^1(\Si,\bar g_{\la,a})}^2$. 

As the model maps are chosen so that \eqref{est:energy-old} holds we can thus use \eqref{claim:step1-2-conseq} to bound 
\beqs
E(\bar u)-4\pi\leq C \la^{-2} +\norm{\tau_{\bar g_{\mu,b}}(\bar u)}_{L^2(\Si,\bar g_{\mu,b})}\norm{P_{\bar u}\bar w}_{L^2(\Si,\bar g_{\la,a})}
+
C\norm{
\bar w }_\honebarmb^2\\
\leqs 
\norm{\na \EE_b(u,g_{\mu,b})}_{L^2(\Si,g_{\mu,b})}^2.
\eeqs 
This establishes the final claim \eqref{est:LE-energy} and hence reduces the proof of the theorem to the proof of the key estimates \eqref{claim:step1} and \eqref{claim:step2} which we carry out in the rest of this section.

\textit{Proof of  \eqref{claim:step1}:}\\
For the proof of this estimate we can closely follow the arguments used in the proof of \cite[Lemma 4.4]{R-Loj-1} which indeed simplify considerably since 
the restriction of $d^2E(\bar z_{\la,a})$ to the $H^1$-orthogonal complement
 $T^\perp_{\bar z} \ZZ$ of $T_{\bar z} \ZZ$  in $ \Gamma^{H^1}(\bar z^*TS^2)$
 is uniformly positive definite instead of just uniformly definite. 
 To apply this we first recall that $\bar w$ is nearly tangential to $S^2$ at $\bar z=\bar u_1$, with \eqref{est:w-hallo-1} ensuring in particular that $\bar w_1=P_{\bar u_1} \bar w$ is so that $\norm{\bar w_1-\bar w}_{H^1(\Si, \gbarla )}=o(1) \norm{\bar w}_{H^1(\Si, \gbarla )}$. As $\bar z=\RR\bar z_{\la,a}$ was chosen so that 
$\norm{\bar u- \RR \bar z_{\la,a}}_{H^1(\Si,\bar g_{\la,a})}$ is minimal, we furthermore get that $\bar w=\bar u-\bar z$ is nearly $H^1(\Si,\gbarla)$-orthogonal to $T_{\bar z}\ZZ$, i.e. so that the $H^1$-orthogonal projection satisfies 
$\norm{P_{H^1}^{T_{\bar z}\ZZ} \bar w}_{H^1(\Si, \gbarla )}=o(1) \norm{\bar w}_{H^1(\Si, \gbarla )}$,
see \cite{R-Loj-1} for details. 

Thus also $\norm{P_{H^1}^{T_{\bar z}\ZZ} \bar w_1}_{H^1(\Si, \gbarla )}=o(1) \norm{\bar w}_{H^1(\Si, \gbarla )}$, allowing us to exploit the positive definiteness of $d^2E$ at $\bar u_1=\bar z$ in the direction $\bar w_1-P_{H^1}^{T_{\bar z}\ZZ} \bar w_1$ as well as that $\norm{\bar w_1-\bar w}_{H^1(\Si, \gbarla )}=o(1) \norm{\bar w}_{H^1(\Si, \gbarla )}$ to deduce
\beq
\label{est:step1-1-hallo}
\norm{\bar w}_{H^1(\Si,\bar g_{\la,a})}^2 \leqs d^2E(\bar u_1)(\bar w_1,\bar w_1).
\eeq
We now use Remark \ref{rmk:E-var} and $\norm{\bar w}_{L^\infty}+\norm{\bar w}_{H^1(\Si, \gbarla )}=o(1)$ to write 
\beqa \label{est:step1-2-hallo}
dE(\bar u_1)(\bar w_1)-dE(\bar u_0)(\bar w_0)&=\int_0^1\ddt (dE(\bar u_t)(\bar w_t)) dt=\int_0^1d^2E(\bar u_t)(\bar w_t,\bar w_t) dt+\int_0^1 dE(\bar u_t)(\ddt\bar w_t) dt\\
&=d^2E(\bar u_1)(\bar w_1,\bar w_1)+o(1)  \norm{\bar w}_{H^1(\Si, \gbarla )}^2+\int_0^1 \err_t dt
\eeqa
for
$\err_t=dE(\bar u_t)(\ddt\bar w_t)= dE(\bar u_t)(P_{\bar u_t}(\ddt\bar w_t))=\int_\Si \na \bar u_t \na (P_{\bar u_t}(\ddt\bar w_t))$. 
We note that this error term is also of lower order since 
 $\bar w$ is nearly tangential to $T_{\bar u_t}S^2$ so the variation of $\bar w_t=\abs{\bar u+t\bar w}^{-1}P_{\bar u_t} \bar w$ is essentially normal. To be more precise, a short calculation and \eqref{est:na-z-1} allow us obtain pointwise bounds of 
$\abs{ \na (P_{\bar u_t}(\ddt\bar w_t))}\leqs \bar\rho_{\mu,b} \abs{\bar w}^3+ \abs{\na \bar w} \abs{\bar w}^2, $
which, as $\norm{\bar w}_{L^\infty(\Si)}=o(1)$, yields $\abs{\err_t}\leqs o(1) \norm{\bar w}_{H^1(\Si, \gbarla )}^2$.

Combined, \eqref{est:step1-1-hallo}, \eqref{est:step1-2-hallo} and \eqref{est:energy-old-var} hence yield
\beqa
\label{est:step1-1}
\norm{\bar w}_{H^1(\Si,\bar g_{\la,a})}^2 &\leqs dE(\bar u_1)(\bar w_1)-dE(\bar u_0)(\bar w_0)=dE(\bar z)(P_{\bar z}(\bar w))-dE(\bar u)(P_{\bar u}(\bar w))\\
&\leqs \la^{-2}(\log\la)^\half 
\norm{P_{\bar z}(\bar w)}_{H^1(\Si,\bar g_{\la,a})}+
\norm{\tau_{\bar{g}_{\mu,b}}(\bar{u}) }_{L^2(\Si, \bar{g}_{\mu,b})}\norm{\bar w}_{L^2(\Si, \bar{g}_{\mu,b})}, 
\eeqa 
which, as $\gbarmb \sim \gbarla$ and $\norm{P_{\bar z}(\bar w)}_{H^1(\Si,\bar g_{\la,a})}\sim  \norm{\bar w}_{H^1(\Si,\bar g_{\la,a})} $, implies the key estimate
\eqref{claim:step1}.

For the proof of the second key estimate \eqref{claim:step2}, we now 
crucially use that 
our family of metrics $\GG_b$ provides an efficient way of encoding the formation of a bubble.

\newcommand{\zmblaeps}{z_{\mu_\eps,b}^{\la_\eps,a_\eps}}
\newcommand{\gmbep}{g_{\mu_\eps,b}}
\textit{Proof of \eqref{claim:step2}:} \\
For this proof we work in the alternative viewpoint, so 
pull-back all maps by the \textit{fixed} diffeomorphism $T_{\mu,b}$, i.e. set 
$u:=\bar u\circ T_{\mu,b}$, $z:=\bar z\circ T_{\mu,b}$
$w:= u-z=\bar w\circ  T_{\mu,b}$, interpolate  using 
 $ u_t:=\bar u_t\circ T_{\mu,b}=\pi_{S^2}(u+tw)$, set $w_t := \partial_t u_t = \abs{u+tw}^{-1} P_{u_t}(w)$ and work with respect to
$\gmb=T_{\mu,b}^*\gbarmb\in \GG_b$. 

We now use the variation $(z_\eps,g_\eps)=(\zmblaeps,\gmbep)$ from Lemma \ref{lemma:special-variation} for which we know that  $ \norm{\peps z_\eps}_{H^1(\Si,g)} + \norm{\peps z_\eps}_{L^{\infty}(\Si)}\leqs \la^{-2}$
and that 
\beqa \label{est:step2-2}
\la^{-3} &
\sim \ddeps E(z_\eps,g_\eps)=dE(z,g)(\peps z_\eps,\peps g)= dE(u,g)(\peps z_\eps,\peps g)+ \int_{0}^1 \ddt [ dE(u_t,g)(\peps z_\eps,\peps g) ]dt \\
&= dE(u,g)(\peps z_\eps,\peps g)+
d^2E(z,g)((\peps z_\eps,\peps g),(P_z w,0)) +\err_2
\eeqa
where we split the error term 
$\err_2:= \int_0^1 e_z(t)+e_g(t) dt$ 
into the contribution of $\peps z_\eps$ and $\peps g$. 

Since the second variation of $E_g(\cdot):=E(\cdot,g)$ is 
$d^2E_g(u_t)(v_1,v_2) =\int_{\Si} \nabla_g v_1 \nabla_g v_2 - \abs{\nabla_g u_t}^2 v_1v_2 \,dv_{g}
$ we immediately see that the last two terms in  
\beqas
\abs{e_z(t)}&= \abs{\ddt( dE_g(u_t)(\peps z_\eps))-d^2E_g(z)(\peps z_\eps,P_z w)}= \abs{\ddt( dE_g(u_t)(P_{u_t}\peps z_\eps))-d^2E_g(z)(\peps z_\eps,P_z w)}
\\
&\leq \abs{dE_g(u_t)(\ddt P_{u_t}(\peps z_\eps))}+ \abs{d^2E_g(u_t)(P_{u_t}\peps z_\eps,w_t)}+ \abs{d^2E_g(z)(\peps z_\eps, P_z w)}
\eeqas 
are both bounded by 
$ C \norm{\peps z_\eps}_{H^1(\Si,g)} \norm{w}_{H^1(\Si,g)}+C \norm{\peps z_\eps}_{L^\infty(\Si,g)} \norm{w}_{H^1(\Si,g)}^2\leqs \la^{-2} \norm{w}_{H^1(\Si,g)}$. 
\\
As  
$ dE_g(u_t)(\ddt P_{u_t}(\peps z_\eps))
=\int_\Si \na_g u_t \na_g (P_{u_t}(\ddt P_{u_t}(\peps z_\eps))dv_g $ 
and as 
 $P_{u_t}(\ddt (P_{u_t} \ppeps z_\eps)) = -\langle \peps z_\eps, u_t \rangle w_t$, it is easy to see that the same bound also applies to this term, thus ensuring that 
$  \abs{e_z(t)} \leqs  \la^{-2} \norm{w}_{H^1(\Si,g)}$. 

As  $P_z w=\ddelta u_{1+\de}$ we can furthermore write 
\beqas
e_g(t)&:= \ddt( dE(u_t,g)(0,\peps g))-d^2E(z,g) ((0,\peps g),(P_z w,0))\\
&=\ddelta \big[ ( dE(u_{t+\de},g)-dE(u_{1+\de},g))(0,\peps g)\big] 
 = -\thalf \langle \ddelta (k(u_{t+\de},g)-k(u_{1+\de},g)),\peps g\rangle_{L^2(\Si,g)}\eeqas
and hence use that $\norm{\peps g}_{L^\infty(\Si,g)}\sim \norm{\pmu g_\mu}_{L^\infty(\Si,g)}\sim \mu^{-1}\sim \la^{-1}$ to bound
\beqas
\label{def:metric-error}
\abs{e_g(t)}&\leqs  
\la^{-1} 
\norm{\ddelta (k(u_{t+\de},g)-k(u_{1+\de},g))}_{L^1(\Si,g)}
\leqs \la^{-1} \int_{\Si} \abs{\na_g z}_g^2\abs{w}^2+\abs{\na_g w}_g^2 
dv_g \leqs \la^{-1}\norm{w}_{H^1(\Si,g)}^2
.\eeqas 
As $\EE_b$ is the restriction of the energy to $H^1(\Si,S^2)\times \GG_b$ which contains $(u_\eps,g_\eps)$ 
we can write the first term $T_1:= dE(u,g)(\peps z_\eps,\peps g)$ in 
\eqref{est:step2-2} as 
$T_1=\langle \na\EE_b(u,g),\peps (z_\eps,g_\eps)\rangle_{L^2(\Si,g)}$ and use \eqref{est:special-var} to bound 
\beqas 
\label{est:step2-t-1}
\abs{T_1}\leq  
 \norm{\na \EE_b(u,g)}_{L^2(\Si,g)} \norm{\peps( z_\eps,g_\eps)}_{L^2(\Sigma,g)}
 \leqs \la^{-2}  \norm{\na \EE_b(u,g)}_{L^2(\Si,g)}. 
\eeqas 
Similarly, 
applying \eqref{est:sec-var-map} and using that $\norm{P_z w}_{H^1(\Si,g)}\sim\norm{w}_{H^1(\Si,g)}$ while $\abs{\peps a_\eps}\leqs \la^{-2}$ 
allows us to bound the second term $T_2:= d^2E(z,g)((\peps z_\eps,\peps g),(P_z w,0))$ in \eqref{est:step2-2} by 
\beqas 
\label{est:step2-t-2}
\abs{T_2}\leqs 
\big[ 
\la^{-3} (\log \la)^{\half}\abs{\peps \la}+\la^{-1} \abs{\peps a}+\la^{-3} \abs{\peps \mu_\eps}\big]
\norm{w}_{H^1(\Si,g)}\leqs \la^{-3} (\log \la)^{\half}\norm{w}_{H^1(\Si,g)}.
\eeqas 
Inserting these bounds into
\eqref{est:step2-2}
hence yields 
\beqas
\la^{-3}\leqs \la^{-2}\norm{\na \EE_b(u,g_{\mu,b})}_{L^2(\Si,g)}+
\la^{-3} (\log \la)^{\half} \| w \|_{H^1(\Si, g)}+ \la^{-2}\norm{w}_{H^1(\Si, g)}+ \la^{-1} \| w \|_{H^1(\Si, g)}^2,
\eeqas
which immediately yields our second key estimate \eqref{claim:step2}, thus completing the proof of the theorem. 
\end{proof}

\section{Proof of the technical lemmas used in Section \ref{sec:proof-LojEst}}
\label{sec:technical-proofs}

\subsection{Precise definition and key properties of the maps $\bar z_{\la,a}$ }\label{subsec:sing-models}
$ $ \\
We begin by recalling how the maps 
 $\zbarla$  were constructed as modifications of the maps 
 $\pi_\la \circ F_a$
 in \cite{MRS,R-Loj-1}.
 To this end we first recall that the leading order term in the expansion
$$\pi_\la(x)-\pi(\infty)=\left( \tfrac{2\la x}{1+\abs{\la x}^2}, \tfrac{2}{1+\abs{\la x}^2} \right)=\mfrac{2}{\la}\big(\mfrac{x}{\abs{x}^2},0\big)+O(\la^{-2}) \text{ for } \abs{x}\sim 1$$
can be viewed as a multiple of the derivative of the principal part $-\log\abs{x-y}$ of Green's function. To be more precise, as $G:\Si\times \Si\to \R$ is
characterised by $-\Delta_pG(p,a)=2\pi (\de_a-(\Area(\Si,g))^{-1})$, we can write the function 
$G_a:=G\circ F_a^{-1}$, which represents $G$ in the usual conformal coordinates around $a$, as 
 $G_a(x,y)=-\log\abs{x-y}+J_a(x,y)$ for a 
smooth function $J_a:\DD_\iota\times \DD_\iota\to \R$, called the regular part of Green's function. 

Setting  $\hatja(x) = (\nabla_y J_a(x,0) - \na_y J_a(0,0),0)$ 
we can hence use that for $x\in \DD_\iota$ with $\abs{x}\sim 1$ the function 
\beq \label{def:v-hat} \hat v_{\la,a}(x):=\pi_\la(x)+\mfrac{2}\la \hatja(x) \eeq 
 agrees upto errors of order $O(\la^{-2})$ (in $C^k$) with the function 
$\pi(\infty)+\frac{2}{\la} (\na_y G_a(x,0)-\na_yJ_a(0,0),0)$ 
that represents  
$$H_{\la,a} (p):=\pi(\infty)+\mfrac2\la (\na_aG(p, a)-\na_yJ_a(0,0),0)$$ in local coordinates. 
We note that  $H_{\la,a}$ 
is defined and harmonic on all of $\Si\setminus \{a\}$, allowing us to  obtain suitable models $\zbarla$ by interpolating and projecting onto $S^2$, i.e. by setting 
\beq
\label{def:maps-v}
\bar v_{\la,a}:= \phi^a \cdot \hat v_{\la,a}
\circ F_a+(1-\phi^a)H_{\la,a} \text{ and } \bar z_{\la,a}:=\pi_{S^2}(\bar v_{\la,a})=\mfrac{\bar v_{\la,a}}{\abs{\bar v_{\la,a}}}.
\eeq
Here $\phi^a\in C_c^\infty( B_{R_\Si(r_0)}(a),[0,1])$ is a cut-off function 
with $\ph^a \equiv 1$ on $B_{R_\Si(\half r_0)}(a)$ that is obtained by extending  $\phi\circ F_a$ by zero for a fixed 
$\ph \in C_{c}^{\infty}(\DD_{r_0},[0,1])$  with $\ph \equiv 1$ on $\DD_{\half r_0}$.

\begin{rmk}\label{rmk:coord-normal}
 As in  \cite{MRS,R-Loj-1} we normalise our choice of coordinate charts $F_a$ so that the transition maps $F_{b,a}$ are given by the Euclidean translations $F_{b,a} (x)=x- F_b(a)$ if $\gamma=1$ and the hyperbolic translations $F_{b,a}(x)=\tfrac{x_1+\text{i} x_2-F_b(a)}{1-\overline{F_b(a)}(x_1+\text{i} x_2)}$ if  $\gamma\geq 2$, noting that this always allows us to write 
\beq
\label{def:err-ab}
F_{b,a} (x)=x- F_b(a)+\err_{b,a}(x)
\eeq
for an error term that satisfies 
\beq
\label{est:err-bound}
\abs{\err_{b,a}(x)} + \abs{x} \abs{\nabla_x \err_{b,a}(x)} \leqs d_{a,b} \abs{x}^2 + d_{a,b}^2\abs{x} \text{ and } \abs{\nabla_x^2 \err_{b,a}(x)} \leqs d_{a,b},
\eeq
for $d_{a,b}=\dist_{g_\Si}(a,b)$ and 
\beq
\label{est:err-ab-errb}
\abs{\na_a \err_{b,a}(x)}+\abs{x}\abs{\na_x \na_a \err_{b,a}(x)} \leqs \abs{x}^2+ d_{a,b} \abs{x}.
\eeq 
\end{rmk}
The construction of these maps ensures that  
\beq
\label{eq:map-inside}
\bar z_{\la,a}=\hat z_{\la,a}\circ  F_a  \text{ on } B_{R_\Si(\half r_0)}(a) 
\text{ for } \hat z_{\la,a}:= 
\pi_{S^2}(\hat v_{\la,a})=\pi_{S^2}(\pi_\la+\mfrac2\la \hatja(x)) 
\eeq
as well as that for any fixed $R>0$
\begin{eqnarray}
\label{est:map-var-outside}
    \| \bar{z}_{\la,a} -\pi(\infty)\|_{C^1(\Si \setminus B_R(a))}\leqs \la^{-1
} &\text{ and }&
\| \ppeps \bar{z}_{\la_\eps,a_\eps} \|_{C^1(\Si \setminus B_R(a))} \leqs \la^{-2}\abs{\peps \la} + \la^{-1}\abs{\peps a},\\
\label{est:laplace-var-outside}
\| \Delta_{g_\Si} \bar v_{\la,a} \|_{L^{\infty}(\Si \setminus B_R(a))} \leqs \la^{-2} &\text{ and } &\| \ppeps \Delta_{g_\Si} \bar v_{\la_\eps,a_\eps} \|_{L^{\infty}(\Si \setminus B_R(a))} \leqs \la^{-3} \abs{\peps \la}+ \la^{-2} \abs{\peps a}.
\end{eqnarray}
These estimates in particular ensure that  
\beq
\label{est:tension-var-outside}
\| \tau(\bar z_{\la,a})  \|_{L^{\infty}(\Si \setminus B_R(a))}\leqs \la^{-2} \text{ and }
\| \peps \tau(\bar z_{\la_\eps,a_\eps})  \|_{L^{\infty}(\Si \setminus B_R(a))} \leqs \la^{-3} \abs{\peps \la}+ \la^{-2} \abs{\peps a}.
\eeq
Writing $$\tau(\hat z_{\la,a}) =P_{\hat z_{\la,a}}(\Delta \pi_{S^2}(\hat v_{\la,a}))= P_{\hat z_{\la,a}}(d^2\pi_{S^2}(\hat v_{\la,a})(\nabla \hat v_{\la,a},\nabla \hat v_{\la,a}) + d\pi_{S^2}(\hat v_{\la,a})(\Delta \hat v_{\la,a}))$$ 
and using that $ \widehat{j_{a}} $  
is a harmonic function which vanishes at $x=0$ one can also see that 
\beqa
\label{est:tension-disks-RL1}
\abs{\tau(\hat z_{\la,a})(x)} &\leqs \lambda^{-2} + \lambda^{-1} \abs{x} \abs{\nabla \pi_\la(x)}^2 + \lambda^{-1} \abs{\nabla \pi_\la(x)}(1+\lambda \abs{x})^{-1}\sim \la^{-2}+ (1+\la \abs{x})^{-3} ,
\eeqa compare also 
 \cite[Estimate (3.49)]{R-Loj-1} as well as that 
\beqa
\label{est:tau-avar} 
\left|
\peps \tau (\hat z_{\la,a_\eps} ) \right| \leqs & \; \abs{\peps \hat v_{\la,a_\eps} }(\abs{\nabla \hat v_{\la,a}}^2+ \abs{\Delta \hat v_{\la,a}}) + \abs{\nabla \peps \hat v_{\la,a_\eps}} \abs{\nabla \hat v_{\la,a}} 
\leqs \; \la^{-1} (\abs{x} 
\abs{\nabla \pi_{\la}}^2 
+ \abs{\nabla \pi_{\la}}) \leqs \la^{-1}\abs{\na \pi_\la}
\eeqa
whenever $\abs{\peps a}\leqs 1$. 
In particular 
\beq 
\label{est:tau-weighted} 
\norm{\tau(\hat z_{\la,a})(1+\la \abs{x})}_{L^2(\DD_{r_1})}
+\norm{\peps \tau(\hat z_{\la,a_\eps}) \abs{\na \pi_\la}^{-1}}_{L^2(\DD_{r_1})}
\leqs \la^{-1},
\eeq
noting that the second term is simply the expression for $\norm{\peps \tau_{\bar g_{\la,a}}(\bar z_{\la,a_\eps}\circ F_{a_\eps,a})}_{L^2(B_{R_1}(a),\bar g_{\la,a})}$ in local coordinates. 

\subsection{Variations of $E(\bar z_{\la,a})$ with respect to $a$}\label{subsec:dEda}$ $ \\
Here we briefly explain how the above estimates imply that $ \frac{d}{d \eps} E(\bar{z}_{\la,a_\eps})\leqs \la^{-2}$ whenever
$\abs{\peps a} = 1$, in order to establish the estimate
\eqref{est:claim-peps-za} used in the previous section. 
To see this we fix a function 
$\psi \in C_c^{\infty}(\DD_{\frac12r_0},[0,1])$ with $\psi \equiv 1$ on $\DD_{\frac13 r_0}$ and use the cut-off functions  $\psi^{a_\eps}:= \psi \circ F_{a_\eps}\in C_c^{\infty}(B_{R_\Si(\frac12 r_0)}(a_\eps),[0,1]) $ to 
split $$E(\bar{z}_{\la,a_\eps})= \mfrac12  \int_{\Si} (1-\psi^{a_\eps}) \abs{\nabla_{g_\Si} \bar{z}_{\la,a_\eps}}^2 dv_{g_\Si} + \mfrac12 \int_{\Si} \psi^{a_\eps} \abs{\nabla_{g_\Si} \bar{z}_{\la,a_\eps}}^2 dv_{g_\Si}=:I_1(\eps)+I_2(\eps).$$
Since $1-\psi^{a_\eps}$ vanishes on a fixed ball $B_{R}(a)$ (for $\eps$ small) we immediately get from \eqref{est:map-var-outside} that 
$
\abs{\frac{d}{d \eps} I_1}  \leq C \la^{-2}.
$

We now use that $\bar{z}_{\la,a_\eps}=  \hat z_{\la,a_\eps}\circ F_{a_\eps}$ on the support of $\psi^{a_\eps}=\psi\circ F_{a_\eps}$, compare \eqref{eq:map-inside}, and the conformal invariance of $E$ 
to write  
$I_2(\eps)=\thalf \int
\psi \abs{\na \hat z_{\la,a_\eps}}^2 dx$. As  $\peps \hat{z}_{\la,a_\eps}\in T_{\hat z_{\la,a}}S^2$
we hence get that  
\beqa 
\label{eq:ddeps-I2}
\ddeps I_2 &
=
\int \psi \na \peps \hat z_{\la,a_\eps}\na \hat z_{\la,a} dx= -\int \na \psi \peps \hat z_{\la,a_\eps}\na \hat z_{\la,a} dx
-\int  \psi \peps \hat z_{\la,a_\eps}\tau(\hat z_{\la,a}) dx.
\eeqa
As $\hatja$ is smooth with $ \widehat{j_{a_\eps}}(0)=0$ we can bound 
 $\abs{\ppeps \hat z_{\la,a_\eps}(x)} \leqs \la^{-1} \abs{\ppeps \widehat{j_{a_\eps}}(x)} \leqs \la^{-1} \abs{x}\leq \la^{-2}(1+\la \abs{x})$. Combined with
\eqref{est:tau-weighted}
  and the fact that  
$\abs{\nabla \hat z_{\la,a}} \leqs \la^{-1}$ on 
$\supp(\na \psi)$
this allows us to estimate
\beqa 
\abs{\ddeps I_2(\eps)}
&\leqs \la^{-2} +\la^{-2}   \norm{(1+\la\abs{x}) 
\abs{\tau(\hat z_{\la,a})}}_{L^1(\DD_{r_1})}
 \leqs \la^{-2}, 
\eeqa
thus completing the proof of the claimed bound $ \frac{d}{d \eps} E(\bar{z}_{\la,a_\eps})\leqs \la^{-2}$ that we need to establish \eqref{est:claim-peps-za}.

\subsection{Proof of Lemma \ref{lemma:special-variation}}
\label{subsec:proof-lemma-norm-var}
$ $ \\
Here we establish 
 Lemma \ref{lemma:special-variation}, i.e. 
that $\norm{\peps z_\eps}_{\dot{H}^1(\Si,g)} +\norm{\peps z_\eps}_{L^{\infty}(\Si)}\leqs \la^{-2}$ for 
  $z_\eps= \bar z_{\la_\eps,a_\eps}\circ T_{\mu_\eps,b}
 $ determined by parameters  $\mu_\eps$, $\la_\eps$ and $a_\eps$ satisfying \eqref{eqn:v-par-def} and $g=g_{\mu,b}$. 
As  the parameters are assumed to be compatible, compare   \eqref{est:compatible-for-free}, we know that $\Si_0:= \Si \setminus B_{R_\Si(r_1)}(b)$ is disjoint from a fixed sized ball $B_{R}(a)$ around $a$ 
and that  
$B_{R_\Si(r_1)}(b)$ is contained in the ball $B_{R_\Si(\half r_0)}(a_\eps)$ where  \eqref{eq:map-inside} is applicable.

 On $\Si_0$ we have  $T_{\mu,b}\equiv \id$, 
so the maps and metrics are simply given by  
$z_{\eps}= \bar z_{\la_\eps,a_\eps}$  and $g=\gbarmb$ and the claimed bounds trivially hold as 
\eqref{est:map-var-outside} ensures that 
\beqs
\norm{\peps z_\eps}_{\dot{H}^1(\Si_0,g)}
+\norm{\peps z_\eps}_{L^{\infty}(\Si_0)}\leqs 
\norm{\peps \bar z_{\la_\eps,a_\eps}}_{C^1(\Si_0)} \leqs \la^{-2}\abs{\peps \la} + \la^{-1} \abs{\peps a} 
\leqs \la^{-2},
\eeqs
since $\abs{\peps \la}\sim 1$ and $\abs{\peps a}\leqs \la^{-2} $, compare \eqref{est:spe-var-a}. 

To establish the claims on 
 $B_{R_\Si(r_1)}(b)$ where $g_{\mu,b}=(T_\mu\circ F_b)^*g_{\pi,\mu}$ 
 we can work 
in the ($\eps$ independent) coordinates $x=T_\mu(F_b(p))$ in which $z_\eps=\bar z_{\la_\eps,a_\eps}\circ T_{\mu_\eps,b}= \pi_{S^2}\circ \hat v_{\la_\eps,a_\eps}\circ F_{a_\eps} \circ F_b^{-1}\circ T_{\mu_\eps}\circ F_b$
is given by 
\beq
\label{eq:tilde-z-eps}
\tilde z_\eps(x):= 
z_\eps\circ F_b ^{-1}\circ 
T_\mu^{-1}= 
\pi_{S^2}(\tilde w_\eps+\tilde j_\eps ) \text{  for } 
\tilde w_\eps = \pi_{\la_\eps}
\circ F_{b,a_\eps}\circ T_{\mu,\mu_\eps}, \; \tilde j_\eps:= 
\mfrac{2}{\la_\eps} \widehat{j_{a_\eps}}\circ F_{b,a_\eps}\circ T_{\mu,\mu_\eps}  ,\eeq
compare \eqref{def:v-hat} and \eqref{def:maps-v}.
The claimed bound on 
$\norm{\peps z_\eps }_{L^{\infty}}$ thus follows if we prove that $\norm{\peps \tilde w_\eps }_{L^{\infty}}+\norm{\peps \tilde j_\eps }_{L^{\infty}} \leqs \la^{-2}$,
while this and  the fact that 
 $\abs{\na (\tilde w+\tilde j)}\leqs \abs{\nabla \pi_\la\circ F_{b,a}}+\la^{-1}\leqs \abs{\nabla \pi_\la}$, compare \eqref{est:na-z-1}, then
 reduce the proof of the required $H^1$ bound to showing that  $\norm{\na \peps  \tilde w_\eps}_{L^2(\DD_{r_1})}+  \norm{\na \peps  \tilde j_\eps}_{L^2(\DD_{r_1})}\leqs \la^{-2}$. 
 
 As
$\norm{\peps \widehat{j_{a_\eps}}}_{C^1}+\abs{\peps{F_{b,a_\eps}}}\leqs \abs{\peps a_\eps}\leqs \la^{-2}$, and as \eqref{def:Tmu} ensures that 
$\peps T_{\mu,\mu_\eps}$ is of order $O(\la^{-1}) $ in $H^1\cap L^\infty$, the required bounds on $\peps  \tilde j_\eps$ immediately follow.

It hence remains to analyse $\tilde w_\eps=\pi\circ F_\eps$, 
$F_\eps(x):= \la_\eps F_{b,a_\eps}(T_{\mu,\mu_\eps}(x)).
$
As the 
 transition maps of $T_\mu$ and $S_\mu = \mu T_\mu$ are related by 
 $T_{\mu,\tilde \mu}(x) = \tfrac{1}{\tilde \mu}S_{\mu,\tilde \mu}(\mu x)$ 
 we can write 
$F_\eps(x):= \la_\eps F_{b,a_\eps}(\frac{1}{\mu_\eps} S_{\mu,\mu_\eps}(\mu x))
$
and hence use
\beqas
\peps F_\eps=
\tfrac{\la}{\mu}  dF_{b,a}(x)
\peps S_{\mu,\mu_\eps}(\mu x)+ 
\peps [\la_\eps F_{b,a_\eps}(\tfrac{\mu x}{\mu_\eps} )]=:  f_1+f_2
\eeqas
to split $\peps \tilde w_\eps=d\pi(\la F_{b,a}(x)) f_1+ d\pi(\la F_{b,a}(x)) f_2$.
We note that  
\beq \label{est:pi-derivative}
\abs{(d\pi)(\la F_{b,a}(x)}\leqs   (1+\la \abs{x})^{-2} \text{ and }\abs{\nabla((d\pi)(\la F_{b,a}(x)))}\leqs \la (1+\la \abs{x})^{-3}
\eeq
since $\abs{(d\pi)(y)} \sim( 1+\abs{y})^{-2}$,  $\abs{ (d^2\pi)(y)}
\leqs (1+\abs{y})^{-3}$ and since \eqref{est:compatible-for-free} guarantees $1+\abs{\la F_{b,a}(x)}\sim 1 +\la\abs{x}$.

Our choice of variation and Remark \ref{rmk:coord-normal} allow us to write
\beqas 
f_2& =\peps [ \tfrac{\la_\eps}{\mu_\eps}\mu x -\la_\eps F_b(a_\eps)+\la_\eps \err_{b,a_\eps}(\tfrac{\mu x}{\mu_\eps})]=\peps[\la_\eps \err_{b,a_\eps}(\tfrac{\mu x}{\mu_\eps})] \\
& =\peps \la_\eps \err_{b,a}(x)+ \la \na_a \err_{b,a}(x) \peps a-\la \mu^{-1} \peps \mu  \na_x \err_{b,a}(x) x
\eeqas
and hence use \eqref{est:err-bound}, \eqref{est:err-ab-errb} and $\abs{\peps a_\eps}\leqs \la^{-2}$  to bound 
\beqas 
\abs{f_2(x)}
&\leqs (d_{a,b}+\la^{-1})(\abs{x}^2+ d_{a,b}\abs{x}) 
\leqs \la^{-2}\abs{x}(1+\la\abs{x})\\
 \abs{\nabla f_2(x)} &\leqs  \abs{\nabla_x \err_{b,a}(x)} + \la \abs{\nabla_x \nabla_a \err_{b,a}(x)}\abs{\peps a} +  \abs{\nabla_x^2 \err_{b,a}(x)} \abs{x} \leqs \la^{-2}(1+\la\abs{x}).
 \eeqas
Combined with \eqref{est:pi-derivative}
this immediately gives the required $H^1\cap L^\infty$ bounds for $d\pi(\la F_{b,a}) f_2$.

To obtain the analogue bounds on  $d\pi(\la F_{b,a}) f_1$ it now suffices to recall from \eqref{def:Tmu} that 
 \beq
\label{est:s-transition}
\peps S_{\mu,\mu+\eps}(\mu x) = \peps((\mu+\eps) T_{\mu,\mu+\eps}(x)) = x - \phi_{\GG}(x)x = \psi(\log r_1^{-1} \abs{x})x
 \eeq
is supported on $\R^2\setminus \DD_{r_2}$ where both $\abs{(d\pi)(\la F_{b,a}(x)}$ and $\abs{\nabla((d\pi)(\la F_{b,a}(x)))}$
 are of order $O(\la^{-2})$.
 
\subsection{Bounds on the second variation}
\label{subsec:sec-var}$ $ \\
We finally establish  Lemma \ref{lem:sec-var-map}, i.e. show that \eqref{est:sec-var-map} holds for any $ f \in \Ga^{H^1}(z^* TS^2)$ with 
$ \| f \|_{H^1(\Si,g)}=1$ and any variation 
$(z_\eps,g_\eps) = T_{\mu_\eps,b}^* 
(\bar z_\eps=\bar{z}_{\lambda_\eps,a_\eps}, \bar g_\eps=\bar{g}_{\mu_\eps,b})$. 
As all $\bar g_\eps$ are conformal to $g_\Si$ we can  write
\beqas
d^2E(z, g)(\peps (z_\eps, g_\eps),(f,0))=
\ddeps dE(z_\eps,g_\eps)(f,0)=
\ddeps [dE(\bar z_\eps,\bar g_\eps)(f\circ T_{\mu_\eps,b}^{-1},0)]=\ddeps \big[ dE(\bar z_\eps)(\bar f\circ { 
T}_{\mu_\eps,\mu,b})\big] 
\eeqas
for 
 $T_{\mu_\eps,\mu,b}:=  T_{\mu,b}\circ T_{\mu_\eps,b}^{-1}$ and 
$\bar f:= f\circ T_{\mu,b}^{-1}$ and use that  
$\norm{\bar f}_{H^1(\Si,\bar g_{\la,a})}\sim \norm{\bar f}_{H^1(\Si,\bar g_{\mu,b})}=1$, compare \eqref{est:na-z-1}.

As an estimate of 
$\abs{\partial_\la \big[ dE(\bar z_{\la,a} )(\bar f)\big]}  \leqs \la^{-3} (\log\la)^{\half}$ was already established in Lemma 3.11 of \cite{R-Loj-1} it hence suffices to prove that  
\begin{eqnarray}
\abs{\partial_\eps \big[ dE(\bar z_{\la,a_\eps} )(\bar f)\big]} 
&\leqs &\la^{-1} \text{ for } \abs{\peps a_\eps}=1
\label{claim:2ndvar-a}\\
\abs{dE(\bar z_{\la,a} )(\peps \bar f_\eps  )} &\leqs & \la^{-3} \text{ for }\bar f_\eps:= \bar f\circ 
T_{\mu+\eps,\mu,b}
\label{claim:2ndvar-mu}
\end{eqnarray}
To prove \eqref{claim:2ndvar-mu} we
use that $\peps T_{\mu+\eps,\mu,b}$, and hence  
$\peps \bar f_\eps$, is 
supported on
$ B_{R_\Si(r_1)}(b)\subset B_{R_\Si(\half r_0)}(a)$ and work in the coordinates $x=F_a(p)$ in which 
$\bar z$ is given by $\hat z_{\la,a}$ from \eqref{def:v-hat},
while $\bar f_\eps$ is represented by
$\bar f_a\circ \hat T_\eps$ for  $ \hat T_\eps=F_{a,b}^{-1}\circ T_{\mu+\eps,\mu}\circ F_{a,b}$. 
As 
$\ppeps T_{\mu+\eps,\mu}=-\peps T_{\mu,\mu+\eps}$ is described by \eqref{def:Tmu} we have
\beq
\label{est:transition-SV}
\abs{\ppeps \hat T_\eps(x)} \leqs \abs{\ppeps T_{\mu,\mu+\eps}  (F_{a,b}(x))}\leqs 
\mu^{-1} \abs{F_{a,b}(x)} \leqs \la^{-1}(\abs{x}+d_{a,b})\leqs 
\la^{-2}(1+\la\abs{x})
\eeq
allowing us to use \eqref{est:tau-weighted} to bound 
\beq \label{est:proof-2ndvar-mu}
\abs{dE(\bar z_{\la,a} )(\peps \bar f_\eps ) }=
\left| \int_{\DD_{r_1}} \tau(\hat z_{\la,a}) \na \bar f_a \peps \hat T_\eps dx \right| \leq
 \la^{-2} \norm{\na \bar f}_{L^2(\Si,g_\Si)} 
\norm{\tau(\hat z_{\la,a})(1+\la \abs{x})}_{L^2(\DD_{r_1})}\leqs \la^{-3}.
\eeq
To prove
\eqref{claim:2ndvar-a} we argue as in Section \ref{subsec:dEda}, using cut-off functions 
$\psi^{a_\eps}:= \psi \circ F_{a_\eps}\in C_c^{\infty}(B_{R_\Si(\frac12 r_0)}(a_\eps),[0,1]) $ with  $\psi\equiv 1$ on $\DD_{\frac13 r_0}$ 
to split $dE(\bar z_{\la,a_\eps} )(\bar f)=\int_{\Si} \tau_{g_\Si}(\bar{z}_{\la,a_\eps}) \bar f dv_{g_\Si} $ as 
\beqa dE(\bar z_{\la,a_\eps} )(\bar f)= \int_{\Si} (1-\psi^{a_\eps}) \tau_{g_\Si}(\bar{z}_{\la,a_\eps}) \cdot \bar{f} dv_{g_\Si}
+
\int_{\DD_\iota} \psi \tau({\hat z}_{\la,a_\eps}) \cdot\bar{f}_{a_\eps} dx =:
I_1(\eps)+I_2(\eps)
\eeqa
where $\bar f_{a_\eps}=\bar f_a\circ F_{a,a_\eps}$ is the function that represents $\bar f$ in the coordinates $x=F_{a_\eps}(p)$.
Estimate 
 \eqref{est:tension-var-outside} immediately implies that   
$\abs{\ddeps I_1}\leqs \la^{-2}\abs{\peps a_\eps}=\la^{-2}$ while 
 $\abs{\peps F_{a,a_\eps}}\leqs 1$ and \eqref{est:tau-weighted} allow us to bound
\beqa \label{est:2nd-var-proof-a}
\abs{\ddeps I_2(\eps)}&\leqs
\int_{\DD_{r_1}} \abs{\tau(\hat z_{\la,a})} \abs{\nabla \bar f_a} dx+
\int_{\DD_{r_1}} \abs{\peps \tau(
\hat z_{\la,a_\eps})}\abs{\bar{f}_a} dx \\
&
\leqs \norm{\tau(\hat z_{\la,a})}_{L^2(\DD_{r_1})} +\norm{\peps \tau(\hat z_{\la,a_\eps}) \abs{\na \pi_\la}^{-1}}_{L^2(\DD_{r_1})} \norm{ \bar f_a}_{L^2(\DD_{r_1},g_{\pi,\la})} \leqs \la^{-1}.
\eeqa  
This completes the proof of  \eqref{claim:2ndvar-a} and hence the proof of the lemma.

\vspace{-0.1cm}
MR \& SW: Mathematical Institute, University of Oxford, Oxford OX2 6GG, United Kingdom


\begin{thebibliography}{99}
\vspace{-0.3cm}
\bibitem{BMS} A. Bernand-Mantel, C. B. Muratov and T. M. Simon, \emph{A quantitative description of skyrmions in ultrathin ferromagnetic films and rigidity of degree $\pm 1$ harmonic maps from $\mathbb{R}^2$ to $S^2$.} Arch. Rational Mech. Anal. \textbf{239} (2021) 219–299.

\bibitem{Faber-Krahn} L. Brasco,  G. De Philippis and  B. Velichkov, \textit{Faber–Krahn inequalities in sharp quantitative form.} Duke Math. J. \textbf{164} (2015) 1777-1831.

\bibitem{Fusco-Sobolev}
A. Cianchi, N. Fusco, F. Maggi and A. Pratelli, \textit{The sharp Sobolev inequality in quantitative form}. JEMS. \textbf{11}  (2009) 1105-1139.

\bibitem{DeLellis-Mueller} C. De Lellis and S. M\"uller
\textit{Optimal rigidity estimates for nearly umbilical surfaces}. 
JDG \textbf{69} (2005) 75-110.

\bibitem{DIL} B. Deng, R. Ignat and X. Lamy, \emph{The conformal limit for bimerons in easy-plane chiral magnets}. \href{https://arxiv.org/abs/2506.11955}{arxiv:2506.11955}.

\bibitem{Ding-Tian} W. Y. Ding and G. Tian, \emph{Energy identity for a class of approximate harmonic maps from surfaces}.
Comm. Anal. Geom. {\bf 3} (1995) 543–554.


\bibitem{Yamabe} M. Engelstein, R. Neumayer and L. Spolaor,
\textit{Quantitative stability for minimizing Yamabe metrics.}
Trans. Amer. Math. Soc. \textbf{9} (2022) 395--414.

\bibitem{isoperimetric-FMP} A. Figalli, F. Maggi and A. Pratelli, \textit{A mass transportation approach to quantitative isoperimetric inequalities.} Invent. Math.  \textbf{182} (2010) 167-211.

\bibitem{Figalli-Maggi-Pratelli-Sobolev} A. Figalli, F. Maggi and A.  Pratelli, \textit{Sharp stability theorems for the anisotropic Sobolev and log-Sobolev inequalities on functions of bounded variation.} Adv. in Math. \textbf{242} (2013) 80--101.


\bibitem{Figalli-recent} A. Figalli and Y. R.-Y. Zhang, \textit{Sharp gradient stability for the Sobolev inequality.} Duke Math. J. \textbf{171} (2022) 2407-2459.

\bibitem{BMI-1} A. Figalli, P. van Hintum and M. Tiba, \emph{Sharp quantitative stability of the Brunn-Minkowski inequality.}  
\href{https://arxiv.org/abs/2310.20643}{arxiv:2310.20643}.

\bibitem{sharp-isoperimetric} N. Fusco, F. Maggi and A. Pratelli, \textit{The sharp quantitative isoperimetric inequality.} Ann. of Math. \textbf{168} (2008), 941–980.

\bibitem{ZLG1} A. Guerra, X. Lamy and K. Zemas, \emph{Optimal quantitative stability of the M\" obius group of the sphere in all dimensions.} \href{https://arxiv.org/abs/2401.06593}{arxiv:2401.06593}

\bibitem{ZLG2} A. Guerra, X. Lamy and K. Zemas. \emph{Sharp quantitative stability of the Möbius group among sphere-valued maps in arbitrary dimension.} Trans. Amer. Math. Soc. \textbf{378} (2025) 1235-59

\bibitem{H-Z} J. Hirsch and K. Zemas, \emph{A note on a rigidity estimate for degree $\pm1$ conformal maps on $S^2$}.  Bull. Lond. Math. Soc. {\bf 54} (2022) 256-263.

\bibitem{H-R-T-TMF} T. Huxol, M. Rupflin and P. Topping, \emph{Refined asymptotics of the Teichmüller harmonic map flow into general targets.} Calc. Var. Partial Differential Equations {\bf 55} (2016) 1-25.

\bibitem{Lamm-Nguyen} T. Lamm and H. Nguyen, \textit{Quantitive rigidity results for conformal immersions}.
Amer. J. Math. \textbf{136} (2014) 1409--1440.

\bibitem{Lin-Wang} F.-H. Lin and C.-Y. Wang, \emph{Energy identity of harmonic map flows from surfaces at finite singular
time.} Calc. Var. Partial Differential Equations {\bf 6} (1998) 369-380.


\bibitem{Luckhaus} S. Luckhaus and K. Zemas, 
\textit{Rigidity estimates for isometric and conformal maps from $S^{n-1}$ to $\R^n$.}
Invent. Math. \textbf{230} (2022) 375--461 . 

\bibitem{Sobolev-Maggi-Neumayer} F. Maggi, R. Neumayer and I. Tomasetti,
\textit{Rigidity theorems for best Sobolev inequalities}. Adv. in Math. \textbf{434} (2023) Paper No. 109330, 43 pp.

\bibitem{MRS} A. Malchiodi, M. Rupflin and B. Sharp, \emph{\Loj inequalities near simple bubble trees}. Amer. J. Math. \textbf{146} (2024) no. 5, 1361--1397.


\bibitem{MSkyr} A. Monteil, C.B. Muratov, T.M. Simon and V.V. Slastikov, \emph{Magnetic skyrmions under confinement}. Communications in Mathematical Physics \textbf{404} (2023) 1571-1605.

\bibitem{Qian-Tian} J. Qing and G. Tian, \emph{Bubbling of the heat flows for harmonic maps from surfaces.} Comm. Pure
Appl. Math. {\bf 50} (1997) 295-310.

\bibitem{R-TMF-2} M. Rupflin, \emph{Flowing maps to minimal surfaces: Existence and uniqueness of solutions}.  Ann. Inst. H. Poincar\'e C Anal. Non Lin\'eaire,
{\bf 31} (2014) 349–368.


	
\bibitem{R-T-TMHF-1} M. Rupflin and P. Topping, \emph{Flowing maps to minimal surfaces.} Amer. J. Math. {\bf 138} (2016) 1095-1115.


\bibitem{R-Loj-1} M. Rupflin, \emph{\Loj inequalities for almost harmonic maps near simple bubble trees}. Ann. Inst. H. Poincar\'e C Anal. Non Lin\'eaire \textbf{42} (2025) no. 4, 807--850.

\bibitem{R-Loj-2} M. Rupflin,
\textit{Low energy levels of harmonic spheres in analytic manifolds}. Journal of Functional Analysis, \textbf{289} (2025)  Article 111006.


\bibitem{R-Rig-23} M. Rupflin, \emph{Sharp quantitative rigidity results for maps from $S^2$ to $S^2$ of general degree}. \href{https://arxiv.org/abs/2305.17045}{arxiv:2305.17045}.



	
\bibitem{Simon} L. Simon, \emph{Asymptotics for a Class of Non-Linear Evolution Equations, with Applications to Geometric Problems.\/} Ann. of Math. {\bf 118} (1983) 525--571.
		

\bibitem{Struwe} M. Struwe, \emph{On the evolution of harmonic mappings of Riemannian surfaces\/}. Comment. Math. Helv. {\bf 60} (1985) 558--581.

\bibitem{Topping-rigidity} P. Topping, \emph{Rigidity in the harmonic map heat flow}. J. Differential Geometry. {\bf 45} (1997) 593–610.

\bibitem{Topping-HMF-2} P. Topping, \emph{Repulsion and quantization in almost-harmonic maps, and asymptotics of the harmonic map flow}. Ann. of Math. {\bf 159} (2004) 465-534.

\bibitem{Topping-deg-1} P. Topping, \emph{A rigidity estimate for maps from $S^2$ to $S^2$ via the harmonic map flow}. Bull. Lond. Math. Soc. \textbf{55} (2023) 338--343.

\bibitem{Alex-Loj} A. Waldron, \emph{\Loj inequalities for maps of the 2-sphere.} \href{https://arxiv.org/abs/2312.16686}{arxiv:2312.16686}

\end{thebibliography}
\end{document}